\documentclass{emsprocart}

\usepackage{enumerate}
\usepackage{tikz-cd}

\contact[julian.kuelshammer@mathematik.uni-stuttgart.de]{Julian K\"ulshammer\\Institute for Algebra and Number Theory\\University of Stuttgart\\D-70567 Stuttgart\\Germany}

\newtheorem{theorem}{Theorem}[section]
\newtheorem{corollary}[theorem]{Corollary}
\newtheorem{lemma}[theorem]{Lemma}
\newtheorem{proposition}[theorem]{Proposition}

\newtheorem{theoremstar}{Main Theorem}

\newtheorem*{propositionstar}{Proposition} %%%% for unnumbered statements

\theoremstyle{definition}
\newtheorem{definition}[theorem]{Definition}
\newtheorem{remark}[theorem]{Remark}
\newtheorem{example}[theorem]{Example}
\newtheorem{question}[theorem]{Question}

\newcommand{\emphbf}[1]{\emph{\textbf{#1}}}

\newcommand{\Hom}{\operatorname{Hom}}
\newcommand{\Ker}{\operatorname{Ker}}
\newcommand{\End}{\operatorname{End}}
\newcommand{\Ext}{\operatorname{Ext}}
\newcommand{\can}{\operatorname{can}}
\newcommand{\modu}{\operatorname{mod}}
\newcommand{\proj}{\operatorname{proj}}
\newcommand{\Mod}{\operatorname{Mod}}
\newcommand{\EM}{\operatorname{EM}}
\newcommand{\rad}{\operatorname{rad}}
\newcommand{\image}{\operatorname{Im}}
\newcommand{\Cok}{\operatorname{Cok}}
\newcommand{\comp}{\operatorname{comp}}
\newcommand{\adj}{\operatorname{adj}}
\newcommand{\id}{\operatorname{id}}

\newcommand{\GL}{\operatorname{GL}}
\newcommand{\Kleisli}{\operatorname{Kl}}
\newcommand{\Ind}{\operatorname{Ind}}
\newcommand{\CoInd}{\operatorname{CoInd}}
\newcommand{\add}{\operatorname{add}}
\newcommand{\projdim}{\operatorname{projdim}}

\title[In the bocs seat: Quasi-hereditary algebras and representation type]{In the bocs seat:\\Quasi-hereditary algebras and representation type}

\author[Julian K\"ulshammer]{Julian K\"ulshammer\thanks{The author wants to thank his collaborators Agnieszka Bodzenta, Steffen Koenig, Vanessa Miemietz, Sergiy Ovsienko, and Ulrich Thiel for the fruitful joint work and the possibility to include parts of our joint work in this article. Additional thanks go to Agnieszka Bodzenta, Ren\'e Marczinzik, Frederik Marks and Theo Raedschelders for helpful comments on earlier versions of the paper.}}

\begin{document}

\begin{abstract}
This paper surveys bocses, quasi-hereditary algebras and their relationship which was established in a recent result by Koenig, Ovsienko, and the author. Particular emphasis is placed on applications of this result to the representation type of the category filtered by standard modules for a quasi-hereditary algebra. In this direction, joint work with Thiel is presented showing that the subcategory of modules filtered by Weyl modules for tame Schur algebras is of finite representation type. The paper also includes a new proof for the classification of quasi-hereditary algebras with two simple modules, a result originally obtained by Membrillo-Hern\'andez in \cite{MembrilloHernandez-Quasihereditaryalgebras-1994}.
\end{abstract}



\maketitle

\section{Introduction}

Bocses were introduced by Ro\v{\i}ter in 1979. From some point of view, they are generalisations of algebras which have turned out to be most useful in the study of problems related to the representation type of algebras. In particular, they were used in attempts to prove the Brauer--Thrall conjectures and they are at the core of the proof of Drozd's tame and wild dichotomy theorem. It was realised by Bautista and Kleiner \cite{BautistaKleiner-Almostsplitsequences-1990} that under certain assumptions the exact category of modules over a bocs has almost split sequences and can be realised as a subcategory of modules over an algebra. This point of view was soon strengthened by Burt and Butler in  \cite{BurtButler-Almostsplitsequences-1991}, which is one of the key ingredients of our work.

On the other hand, quasi-hereditary algebras were introduced by Scott in 1987 \cite{Scott-Simulatingalgebraicgeometry-1987}, and in the more general framework of highest weight categories allowing infinitely many simple objects by Cline, Parshall, and Scott \cite{ClineParshallScott-Finitedimensionalalgebras-1988}. The most prominent examples of this class of algebras arise in the representation theory of groups and Lie algebras, but examples also include all algebras of global dimension $\leq 2$, in particular hereditary algebras and Auslander algebras. Two of the key examples which have motivated much of the theory of quasi-hereditary algebras are blocks of BGG category $\mathcal{O}$ and Schur algebras of the symmetric groups.
A defining property of quasi-hereditary algebras is the existence of certain factor modules $\Delta(\lambda)$ of the indecomposable projective modules $P(\lambda)$, called standard modules. In the case of BGG category $\mathcal{O}$ these are given by the Verma modules while for Schur algebras these are called Weyl modules. They are often much easier to understand than the corresponding simple factor modules. In several instances, it turned out to be useful to consider the subcategory of all modules having a filtration by standard modules $\mathcal{F}(\Delta)$. It was proven by Ringel in \cite{Ringel-Thecategoryofmodules-1991} that this subcategory also has Auslander--Reiten sequences.

Motivated by the example of the universal enveloping algebra of a Borel subalgebra of a Lie algebra $U(\mathfrak{b})\subseteq U(\mathfrak{g})$, Koenig introduced in \cite{Koenig-ExactBorelsubalgebras-1995} the notion of an exact Borel subalgebra of a quasi-hereditary algebra sharing similar properties. In particular, an analogue of the PBW theorem holds and standard modules for the quasi-hereditary algebra are induced from simple modules for the exact Borel subalgebra. In the course of proving existence of exact Borel subalgebras for all quasi-hereditary algebras, Koenig, Ovsienko, and the author could prove the following intrinsic description of quasi-hereditary algebras via bocses.

\begin{theoremstar}[{\cite[Section 11]{KoenigKulshammerOvsienko-Quasihereditaryalgebras-2014}}]
An algebra $A$ is quasi-hereditary if and only if it is Morita equivalent to the endomorphism ring $R$ of a projective generator of $\modu\mathfrak{B}$ for a directed bocs $\mathfrak{B}$. In this case, the subcategory of modules  filtered by standard modules $\mathcal{F}(\Delta)$ is equivalent to the category of modules over the bocs $\mathfrak{B}$.\\
In particular, there exists an algebra $R$, Morita equivalent to $A$, such that $R$ has an exact Borel subalgebra $B$ with $\mathcal{F}(\Delta)$ equivalent to the category of induced modules from $B$ to $R$.
\end{theoremstar}

In this paper, we provide context, motivation, and give applications and further examples of this result. This in particular includes as a corollary a classification of two-point quasi-hereditary algebras by pairs of natural numbers originally obtained by Membrillo-Hern\'andez in \cite{MembrilloHernandez-Quasihereditaryalgebras-1994}, see Theorem \ref{classification2simples}.

For another application recall that a quasi-hereditary algebra is called \emphbf{(left) strongly quasi-hereditary} if all (left) standard modules have projective dimension $\leq 1$, in other words $\mathcal{F}(\Delta)$ is a hereditary exact category. Regarding representation type, the following is an immediate corollary of the proof of Drozd's tame and wild dichotomy theorem and Main Theorem 1:

\begin{propositionstar}
Let $A$ be a strongly quasi-hereditary algebra. Then, the category of modules filtered by standard modules is either representation-finite, tame, or wild.
\end{propositionstar}

It is expected, but not known, whether an analogue of this result holds for more general quasi-hereditary algebras. Nevertheless, there are partial results classifying representation type of the category filtered by standard modules. Most notably, this has been achieved for blocks of BGG category $\mathcal{O}$ by Br\"ustle, Koenig and Mazorchuk \cite{BrustleKoenigMazorchuk-Thecoinvariantalgebra-2001}. Representation-finite categories $\mathcal{F}(\Delta)$ for blocks of the Schur algebras $S(2,d)$ with parameter $n=2$ have been classified by Erdmann, Madsen and Miemietz \cite{ErdmannMadsenMiemietz-OnAuslanderReitentranslates-2010}. In this direction, in joint work with Ulrich Thiel we obtained the following statement on filtered representation type of tame Schur algebras:

\begin{theoremstar}[K--Thiel 2014]
Let $S_q(n,d)$ be a tame ($q$-)Schur algebra, then there are only finitely many modules up to isomorphisms which can be filtered by Weyl modules.
\end{theoremstar}

The paper is organised as follows. In Section 2, we recall necessary results on finite dimensional algebras and representation type. In particular, we introduce the Gabriel quiver of a finite dimensional algebra and recall Drozd's tame--wild dichotomy theorem. Section 3 is devoted to an introduction to quasi-hereditary algebras. In particular, two of the most prominent examples, blocks of BGG category $\mathcal{O}$ and Schur algebras are introduced. Furthermore, the main theorem of \cite{KoenigKulshammerOvsienko-Quasihereditaryalgebras-2014} about existence of exact Borel subalgebras for quasi-hereditary algebras up to Morita equivalence is recalled. In Section 4, bocses are considered from different points of view. This includes the description as a Kleisli category of a monad, Burt--Butler's theory on right and left algebras,  $A_\infty$-algebras, as well as differential biquivers. The reduction algorithm for differential biquivers is recalled, and applied to several examples of quasi-hereditary algebras. As a corollary of the proof of Drozd's tame--wild dichotomy theorem, this section contains the statement that for strongly quasi-hereditary algebras, the subcategory of modules filtered by standard modules is either tame or wild. In the examples subsection, a new proof of Membrillo-Hern\'andez' classification of quasi-hereditary algebras with two simple modules is obtained. The last section, Section 5, contains the results of joint work with Ulrich Thiel on the representation type of the subcategory of filtered modules for tame Schur algebras. Firstly, the corresponding bocses are described, and secondly the reduction algorithm of Bautista and Salmer\'on together with computer calculations is applied to get the main theorem showing that these subcategories have finite representation type.

Throughout the whole article we assume that $k$ is an algebraically closed field although often weaker assumptions suffice. An algebra will always mean a finite dimensional unital associative algebra unless indicated otherwise. By an $A$-module we usually mean a finite dimensional unital left module and denote the category of all such modules by $\modu A$. The category of all modules (not necessarily finite dimensional) is denoted $\Mod A$. The opposite algebra of $A$ is denoted $A^{op}$, and the category of finite dimensional right $A$-modules is denoted by $\modu A^{op}$. The $k$-duality functor is denoted $D=\Hom_k(-,k)\colon \modu A\to \modu A^{op}$.

\section{Finite dimensional algebras and representation type}

In this section we shortly recall the definition of the Gabriel quiver of an algebra and the notion of representation type of finite dimensional algebras and subcategories of their module category. For a general introduction to the representation theory of finite dimensional algebras, the reader is referred to \cite{AssemSimsonSkowronski-Elements1-2006} or \cite{AuslanderReitenSmalo-Representationtheory-1995}. For an introduction to representation type, the reader can consult \cite[Chapter XIX]{SimsonSkowronski-Elements3-2007} or \cite[Chapters 22 and 27]{BautistaSalmeronZuazua-Differentialtensoralgebras-2009}.

\subsection{The Gabriel quiver of an algebra}

Recall that a quiver $Q=(Q_0,Q_1,s,t)$ is a finite oriented graph with vertex set $Q_0$, set of arrows $Q_1$, and functions $s,t\colon Q_1\to Q_0$ determining the starting (resp. terminal) point of an arrow. Given such a quiver $Q$, its path algebra is the algebra with basis given by the paths in the quiver (including a path of length zero $e_i$ for each vertex $i\in Q_0$) and multiplication given by concatenation of paths. It is an important result of Gabriel, that the representation theory of each algebra can be understood via the representation theory of the path algebra of an associated quiver.

\begin{definition}
Let $A$ be a finite dimensional algebra. Then, the \emphbf{Gabriel quiver} of $A$ is the quiver $Q_A$ with vertices given by the isomorphism classes of simple $A$-modules, and the number of arrows $[S_i]\to [S_j]$ given by $\dim D\Ext^1(S_i,S_j)$. 
\end{definition}

\begin{remark}
The usage of the duality $D$ here is non-standard (and not really necessary since the definition only depends on the dimension of this space). It will turn out to be the right thing later in Keller's $A_\infty$-description of the quiver and relations of an algebra (see Theorem \ref{Kellerreconstruction}).
\end{remark}

Path algebras of quivers share the property that they are hereditary, i.e. that their simple modules have projective dimension $\leq 1$. In order to understand the representation theory of all algebras, it is therefore necessary to divide out certain ideals.

\begin{definition}
Let $Q$ be a finite quiver. Let $kQ^+$ be the ideal of $kQ$ spanned by the arrows. An ideal $I\subseteq KQ$ is called \emphbf{admissible} if there is an integer $m$ such that $(kQ^+)^m\subseteq I\subseteq (kQ^+)^2$.
\end{definition}

The following is Gabriel's description of the Morita equivalence classes of finite dimensional algebras by quotients of path algebras by admissible ideals.

\begin{theorem}[Gabriel's theorem]
Let $A$ be a finite dimensional algebra. Let $Q_A$ be the Gabriel quiver of $A$. Then, there is an admissible ideal $I$ such that $A$ is Morita equivalent to $kQ_A/I$, i.e.  $\modu A\cong \modu kQ_A/I$.
\end{theorem}

\begin{remark}
Given an algebra $A$, it is in general quite hard to explicitly determine this admissible ideal $I$ (even if the in general hard task of giving a classification of the simple $A$-modules as well as determining the dimension of the first extension groups between them is already established).\\
The number of generators from $i$ to $j$ in a minimal set of generators for $I$ is given by $\dim D\Ext^2(S_i,S_j)$ (see \cite[Corollary 1.1]{Bongartz-Algebras-1983}). Keller's $A_\infty$-description of the quiver and relations of an algebra (see Theorem \ref{Kellerreconstruction}) gives a precise description on how to determine the quiver and relations for an algebra given a very detailed knowledge of $D\Ext^1(S_i,S_j)$ and $D\Ext^2(S_i,S_j)$, namely the restriction of the $A_\infty$-structure on the $\Ext$-algebra of the direct sum of all simples. 
\end{remark}

\subsection{Representation type}

In this subsection, we recall basic facts about the representation type of an algebra, which is a rough measure on how difficult it is to classify all the finite dimensional indecomposable modules (and hence by Krull--Remak--Schmidt all its finite dimensional modules). 

\begin{definition}
Let $A$ be a finite dimensional algebra, $k[T]$ be the polynomial ring in one variable, and $k\langle X,Y\rangle$ be the free unital algebra on two variables.
\begin{enumerate}[(i)]
\item $A$ is called \emphbf{representation-finite} if there are only finitely many indecomposable modules up to isomorphism.
\item $A$ is called \emphbf{tame} if it is not representation-finite and for any dimension $d\geq 1$ there are finitely many $A$-$k[T]$-bimodules $N^{(1)},\dots, N^{(m_d)}$, which are free of finite rank as $k[T]$-modules, such that all but finitely many $A$-modules of dimension $d$ are isomorphic to modules of the form $N^{(i)}\otimes_{k[T]} k_{\lambda}$ for some simple $k[T]$-module $k_\lambda$.
\item $A$ is called \emphbf{wild} if there is an $A$-$k\langle X,Y\rangle $-bimodule $N$, free of finite rank as a $k\langle X,Y\rangle $-module, such that $N\otimes_{k\langle X,Y\rangle }-\colon \modu k\langle X,Y\rangle\to \modu A$ preserves indecomposability and reflects isomorphisms, i.e. for each indecomposable $k[X,Y]$-module $M$, $N\otimes M$ is indecomposable, and if for two $k\langle X,Y\rangle$-modules $M,M'$, $N\otimes_{k\langle X,Y\rangle }M\cong N\otimes_{k\langle X,Y\rangle } M'$, then $M\cong M'$.  
\end{enumerate}
\end{definition}

Keeping in mind the classification of simple $k[T]$-modules via Jordan normal form, one often describes tameness as the possibility of classifying the indecomposable modules of each dimension by finitely many $1$-parameter families.\\
On the other hand, wildness is equivalent to having for each finite dimensional algebra $B$ an $A$-$B$-bimodule $N$, finitely generated projective as a $B$-module, such that $N\otimes_B -\colon \modu B\to \modu A$ preserves indecomposability and reflects isomorphisms. Thus, the representation theory of a wild algebra incorporates in some sense the representation theory of every other finite dimensional algebra. It is therefore considered a hopeless endeavour.\\
In the late 1970s Drozd proved the following trichotomy for finite dimensional algebras, see \cite{Drozd-Tameandwild-1977, Drozd-Tameandwild-1979, Drozd-Tameandwild-1980}. (Sometimes representation-finite algebras are considered as a subclass of tame algebras, so that the statement will be a dichotomy. For simplicity of stating some of the results later on, for us it is more convenient to follow the convention of excluding it.) For more accessible accounts of a proof see e.g. \cite{CrawleyBoevey-Ontamealgebras-1988, BautistaSalmeronZuazua-Differentialtensoralgebras-2009}

\begin{theorem}
Every algebra is either representation-finite, tame, or wild.
\end{theorem}

The proof of this theorem rests on the theory of bocs reductions. We will explain parts of it in later sections. Several classes of algebras have been classified according to their representation type. Here, let us only mention the two classes that have served as prototypical examples in the represesentation theory of finite dimensional algebras, namely group algebras and hereditary algebras:

\begin{theorem}[(i) Bondarenko, Drozd {\cite{BondarenkoDrozd-Therepresentationtype-1977}}, (iia) Gabriel {\cite{Gabriel-Unzerlegbaredarstellungen-1972}} (iib) Donovan--Freislich {\cite{DonovanFreislich-Therepresentationtheory-1973}} and Nazarova {\cite{Nazarova-Representationsofquivers-1973}}] Let $k$ be a field of characteristic $p\geq 0$.
\begin{enumerate}[(i)]
\item Let $G$ be a finite group. 
\begin{enumerate}[(a)]
\item The group algebra $kG$ is of finite representation type if and only if $p\nmid |G|$ or $p\big| |G|$ and the $p$-Sylow subgroups of $G$ are cyclic.
\item The group algebra $kG$ is of tame representation type if and only if $p=2 \big| |G|$ and the $p$-Sylow subgroups of $G$ are dihedral, semidihedral, or generalised quaternion.
\end{enumerate}
\item Let $Q$ be a connected finite quiver.
\begin{enumerate}[(a)]
\item The path algebra $kQ$ is of finite representation type if and only if the underlying graph of $Q$ is of  Dynkin type $A$, $D$, or $E$.
\item The path algebra $kQ$ is of tame representation type if and only if the underlying graph of $Q$ is of Euclidean type $\tilde{A}$, $\tilde{D}$, or $\tilde{E}$.
\end{enumerate}
\end{enumerate}
\end{theorem}

In this article, we are not only interested in the representation type of the whole module category, but also in the representation type of subcategories $\mathcal{C}\subseteq \modu A$. For this reason we define the following:

\begin{definition}
Let $A$ be a finite dimensional algebra. Let $\mathcal{C}\subseteq \modu A$ be a full subcategory closed under direct summands and direct sums.
\begin{enumerate}[(i)]
\item $\mathcal{C}$ is of \emphbf{finite representation type} if there are only finitely many indecomposable modules in $\mathcal{C}$ up to isomorphism.
\item $\mathcal{C}$ is \emphbf{tame} if it is not of finite representation type and for any dimension $d\geq 1$ there are finitely many $A$-$k[T]$-bimodules $N^{(1)},\dots,N^{(m_d)}$, which are free of finite rank as $k[T]$-modules, such that all but finitely many modules in $\mathcal{C}$ of dimension $d$ are isomorphic to modules of the form $N^{(i)}\otimes_{k[T]} k_\lambda$ for some simple $k[T]$-module $k_\lambda$.
\item $\mathcal{C}$ is \emphbf{wild} if there is an $A$-$k\langle X,Y\rangle$-module $N$, free of finite rank as a $k\langle X,Y\rangle$-module, such that $N\otimes_{k\langle X,Y\rangle}-\colon \modu k\langle X,Y\rangle\to \mathcal{C}$ preserves indecomposability and reflects isomorphisms.
\end{enumerate}
\end{definition}

Unlike for the whole module category $\modu A$, up to the authors knowledge, there is no general criterion for when such a subcategory admits a tame--wild dichotomy theorem. Also it seems that only very few examples are known. Therefore, the following question arises:

\begin{question}
What are good conditions on a subcategory $\mathcal{C}$ to admit a tame-wild dichotomy theorem?
\end{question}

We will later establish such a result for a particular type of subcategory $\mathcal{C}$, see Corollary \ref{stronglyfilteredtamewild}.

\section{Quasi-hereditary algebras}

Quasi-hereditary algebras were introduced by Scott \cite{Scott-Simulatingalgebraicgeometry-1987}, or in greater generality of highest weight categories by Cline, Parshall, and Scott \cite{ClineParshallScott-Finitedimensionalalgebras-1988}. In this section, we start by recalling two prototypical examples of quasi-hereditary algebras, namely blocks of BGG category $\mathcal{O}$ and Schur algebras and introduce the general theory only afterwards. For a general introduction to quasi-hereditary algebras, the reader is advised to look into \cite{DlabRingel-Themoduletheoreticalapproach-1992, KlucznikKoenig-Characteristictiltingmodules-1999} or \cite[Appendix]{Donkin-TheqSchuralgebra-1998}.

\subsection{Blocks of BGG category $\mathcal{O}$}

In this subsection we introduce one of the prototypical examples for quasi-hereditary algebras, blocks of Bernstein--Gelfand--Gelfand category $\mathcal{O}$. For a general introduction to BGG category $\mathcal{O}$ the reader can consult \cite{Humphreys-Representations-2008}, or (for the $\mathfrak{sl}_2$-case) \cite{Mazorchuk-Lectures-2010}.

Let $\mathfrak{g}$ be a finite dimensional semisimple complex Lie algebra. Fix a triangular decomposition $\mathfrak{g}=\mathfrak{n}^-\oplus \mathfrak{h}\oplus \mathfrak{n}^+$. Denote by $\mathfrak{b}=\mathfrak{n}^-\oplus \mathfrak{h}$ the corresponding Borel subalgebra. The reader not familiar with these concepts can just stick to the example $\mathfrak{g}=\mathfrak{sl}_n$, the $n\times n$-matrices of trace zero and fix $\mathfrak{h}, \mathfrak{n}^+, \mathfrak{n}^-$ to be the set of diagonal, lower triangular, upper triangular matrices, respectively. Let $U(\mathfrak{g})$ be the \emphbf{universal enveloping algebra} of $\mathfrak{g}$ which is a particular infinite dimensional associative algebra satisfying $\Mod \mathfrak{g}\cong \Mod U(\mathfrak{g})$.\\
It is a well-known fact that the category $\modu U(\mathfrak{g})$ is \emphbf{semisimple}, i.e. every finite dimensional module is a direct sum of simple modules. It is thus not of much interest to study $\modu U(\mathfrak{g})$ as an abelian category. On the other hand, the category $\Mod U(\mathfrak{g})$ of all (or even the finitely generated) $\mathfrak{g}$-modules is far too big to understand. There is not even a nice classification of all simple modules, see e.g. \cite[Chapter 6]{Mazorchuk-Lectures-2010}. Several different attempts have been made to define (abelian) categories in between $\modu U(\mathfrak{g})$ and $\Mod U(\mathfrak{g})$, which are at the same time easy enough to understand, but capture enough information of the representation theory of $\mathfrak{g}$. Bernstein--Gelfand--Gelfand category $\mathcal{O}$ turned out to be one of the most powerful of these approaches.

\begin{definition}
\emphbf{Bernstein--Gelfand--Gelfand category $\mathcal{O}$} is the full subcategory of $\Mod \mathfrak{g}$ whose objects $M$ satisfy the following three properties:
\begin{enumerate}[{($\mathcal{O}$}1{)}]
\item $M$ is finitely generated,
\item $M$ is a weight module, i.e. $M=\bigoplus_{\lambda\in \mathfrak{h}^*}M_\lambda$ as a $U(\mathfrak{h})$-module,
\item $M$ is locally $U(\mathfrak{n}^+)$-finite, i.e. for every $m\in M$ the cyclic module $U(\mathfrak{n}^+)m$ is finite-dimensional.
\end{enumerate}
\end{definition}

Typical examples of modules are the \emphbf{Verma modules} which are defined as $\Delta(\lambda):=U(\mathfrak{g})\otimes_{U(\mathfrak{b})} \mathbb{C}_\lambda$, where $\mathbb{C}_\lambda$ is the one-dimensional module on which $U(\mathfrak{n}^+)$ acts as $0$ and $\mathfrak{h}$ acts by the character $\lambda\in \mathfrak{h}^*$. It follows from the Poincar\'e--Birkhoff--Witt theorem that $U(\mathfrak{g})\cong U(\mathfrak{n}^-)\otimes_{\mathbb{C}} U(\mathfrak{h})\otimes_{\mathbb{C}} U(\mathfrak{n}^+)$. In particular, $U(\mathfrak{g})$ is free over $U(\mathfrak{b})$ and $\Delta(\lambda)\cong U(\mathfrak{n}^-)\otimes \mathbb{C}_\lambda$ as a $U(\mathfrak{n}^-)$-module. 

Although, like the Verma modules, modules in BGG category $\mathcal{O}$ tend to be infinite dimensional, one can understand this category as an abelian category by understanding the module categories of certain finite dimensional algebras by the following theorem already proved in \cite[Theorem 3]{BernsteinGelfandGelfand-Acertaincategory-1976}:

\begin{theorem}
Let $\mathfrak{g}$ be a finite dimensional complex semisimple Lie algebra. Then there is a block decomposition $\mathcal{O}=\bigoplus_{\chi} \mathcal{O}_\chi$ where $\chi$ ranges over the central characters of the form $\chi_\lambda$. For every $\chi$ there is a finite dimensional algebra $A_\chi$ such that $\mathcal{O}_\chi\cong \modu A_\chi$.
\end{theorem}

Let us give the smallest possible example of such algebras, namely the cases $\mathfrak{g}=\mathfrak{sl}_2$. For different approaches on how to compute these examples (and also the next cases in difficulty) see \cite{Marko-Algebraassociated-2006, Stroppel-CategoryO-2003}.

\begin{example}\label{examplesl2}
For $\mathfrak{g}=\mathfrak{sl}_2$ there are two Morita equivalence classes of blocks for BGG-category $\mathcal{O}$. A block can either be Morita equivalent to $\mathbb{C}$ or to the path algebra of $\begin{tikzcd}1\arrow[yshift=0.5ex]{r}{a}&2\arrow[yshift=-0.5ex]{l}{b}\end{tikzcd}$  with relation $ba$.
\end{example}

The representation type of blocks of BGG category $\mathcal{O}$ has been determined independently by Futorny, Nakano, and Pollack \cite{FutornyNakanoPollack-Representationtype-2001}, and Br\"ustle, Koenig, and Mazorchuk \cite{BrustleKoenigMazorchuk-Thecoinvariantalgebra-2001} using different methods. To state this result we first need to recall notation. By the classification of finite dimensional semisimple complex Lie algebras, each $\mathfrak{g}$ comes equipped with a root system $\Phi$ of one of the Dynkin types $A$--$G$ and a Weyl group $W$. One can choose a set of simple roots which determines a set of positive roots $\Phi^+$ and a set of integral weights $X$ and dominant weights $X^+$. It follows from a result of Soergel, see \cite[Theorem 11]{Soergel-KategorieO-1990} (cf. \cite[Theorem 13.13]{Humphreys-Representations-2008}) that one can restrict attention from arbitrary weights $\lambda\in \mathfrak{h}^*$ to integral weights $\lambda\in X$ since every non-integral block is equivalent to an integral one. Every integral block contains a unique anti-dominant weight $\lambda$ (for the definition see e.g. \cite[p. 54]{Humphreys-Representations-2008}). Call $W_0$ the stabiliser subgroup of $W$ fixing $\lambda$. Also $W_0$ corresponds to a root system $\Phi_0$ of one of the Dynkin types $A$--$G$. 

\begin{theorem}
Let $\mathcal{O}_\chi$ be a block of BGG category $\mathcal{O}$. Let $A_\chi$ be the finite dimensional algebra such that $\modu A_\chi\cong \mathcal{O}_\chi$. 
\begin{enumerate}[(i)]
\item The algebra $A_\chi$ is representation-finite if and only if one of the following three cases occurs: $\Phi_0=\Phi$ and arbitrary $\mathfrak{g}$, $\mathfrak{g}=\mathfrak{sl}_2$ and $\Phi_0=\emptyset$, or $\mathfrak{g}=\mathfrak{sl}_3$ and $\Phi_0=A_1$.
\item The algebra $A_\chi$ is tame if and only if one of the following two cases occurs: $\mathfrak{g}=\mathfrak{sl}_4$ and $\Phi_0=A_2$, or $\Phi=B_2$ and $\Phi_0=A_1$.
\end{enumerate}
\end{theorem}

The paper of Br\"ustle, K\"onig, and Mazorchuk also contains a statement for the representation type of a  subcategory, namely the subcategory of all modules admitting a Verma flag. Here, a module $M$ is said to \emphbf{have a Verma flag} if there is a filtration $0=N_0\subset N_1\subset \dots\subset N_t=M$ for some $t$ with $N_{i+1}/N_i\cong \Delta(\lambda^{(i)})$ for some $\lambda^{(i)}$. Denote by $\mathcal{F}(\Delta)$ the corresponding subcategory of $\modu A_\chi$ for a block $\mathcal{O}_\chi$.

\begin{theorem}
Let $\mathcal{O}_\chi$ be a block of BGG category $\mathcal{O}$. Let $A_\chi$ be the finite dimensional algebra such that $\modu A_\chi\cong \mathcal{O}_\chi$. Let $\mathcal{F}(\Delta)$ be the subcategory of $\modu A_\chi$ corresponding to the objects of $\mathcal{O}_\chi$ which can be filtered by Verma modules.
\begin{enumerate}[(i)]
\item Apart from the cases where $A_\chi$ is re\-pre\-sen\-ta\-tion-fi\-nite, $\mathcal{F}(\Delta)$ is re\-pre\-sen\-ta\-tion-fi\-nite if  and only if $A_\chi$ is tame, or $\Phi=A_4$ and $\Phi_0=A_3$.
\item $\mathcal{F}(\Delta)$ is tame if and only if one of the following three cases occurs: $\Phi=A_1\times A_1$ and $\Phi_0=\emptyset$, $\Phi=A_5$ and $\Phi_0=A_4$, or $\Phi=B_3$ and $\Phi_0=B_2$.
\end{enumerate}
\end{theorem}

Let us close this section by the warning that the reader should keep in mind that although BGG category $\mathcal{O}$ is a very interesting category, it is quite different from $\Mod\mathfrak{g}$. The following two well-known observations give a hint in this direction:

\begin{remark}
\begin{enumerate}[(i)]
\item The following example is taken from \cite[Exercise 3.1]{Humphreys-Representations-2008}. BGG category $\mathcal{O}$ is not extension closed in $\Mod\mathfrak{g}$. Consider the module $U(\mathfrak{g})\otimes_{U(\mathfrak{b})} \mathbb{C}^2$ where $\mathbb{C}^2$ is given a $U(\mathfrak{h})$ action by letting $h\in \mathfrak{h}$ act by the Jordan block $\begin{pmatrix}\lambda &1\\0&\lambda\end{pmatrix}$. It is a self-extension of the Verma module $\Delta(\lambda)$, which is not contained in BGG category $\mathcal{O}$.
\item No projective $U(\mathfrak{g})$-module is contained in BGG category $\mathcal{O}$. This easily follows from the PBW theorem since the restriction of a projective $U(\mathfrak{g})$-module to $U(\mathfrak{h})$ is projective. In contrast, the restriction of a module in BGG category $\mathcal{O}$ to $U(\mathfrak{h})$ is a direct sum of one-dimensional modules.\footnote{The author wants to thank Jeremy Rickard for communicating this nice argument to him via http://math.stackexchange.com/q/1418241}
\end{enumerate}
One approach to come a bit closer to $\Mod\mathfrak{g}$ is to consider fat category $\mathcal{O}^{[n]}$, where $(\mathcal{O}2)$ is replaced by a decomposition into modules where $U(\mathfrak{h})$ acts via a Jordan block of size smaller than or equal to $n$.
\end{remark}

\subsection{Schur algebras and $q$-Schur algebras}

Another prominent example in the theory of quasi-hereditary algebras is the example of the Schur algebra, which is closely related to the symmetric group. We also include its cousin, the $q$-Schur algebra, responding to a question of Stephen Donkin at the AMS-EMS-SPM joint meeting 2015 in Porto. For a general introduction to Schur algebras, see e.g. \cite{Green-Polynomialrepresentations-2007}. For the $q$-Schur algebra, see e.g. \cite{Donkin-TheqSchuralgebra-1998}.

Let $V=k^n$ be an $n$-dimensional vector space. Then, $V^{\otimes d}$ becomes a $k\GL(n)$-$k\mathfrak{S}_d$-bimodule. The left action of the general linear group $\GL(n)$ is the diagonal action given by $g\cdot (v_1\otimes\cdots \otimes v_d)=gv_1\otimes \cdots\otimes gv_d$ and the right action of the symmetric group $\mathfrak{S}_d$ is given by permuting the tensor factors as $(v_1\otimes\cdots \otimes v_d)\cdot \sigma=v_{\sigma^{-1}(1)}\otimes \cdots\otimes v_{\sigma^{-1}(d)}$. Schur--Weyl duality asserts that the two maps $k\GL(n)\to \End_{k\mathfrak{S}_d}(V^{\otimes d})$ and $k\mathfrak{S}_d\to \End_{k\GL(n)}(V^{\otimes d})$ induced by the corresponding representations  are surjective. This motivates the definition of the following algebra, which connects the representation theories of the symmetric group and the general linear group.

\begin{definition}
Let $\mathfrak{S}_d$ be the symmetric group on $d$ letters. Let $V=k^n$ be an $n$-dimensional vector space. Then, the \emphbf{Schur algebra} $S(n,d)$ is defined as $S(n,d):=\End_{k\mathfrak{S}_d}(V^{\otimes d})$, 
\end{definition}

There is a strong relationship between the representation theory of the Schur algebra and the representation theory of the symmetric group if the characteristic is not too small. This relationship was already observed by Schur in 1901.

\begin{theorem}
Let $k$ be a field of characteristic $p\geq 0$. Let $S(n,d)$ be the Schur algebra. Then, there is an idempotent $e\in S(n,d)$ with $eS(n,d)e\cong k\mathfrak{S}_d$. The corresponding functor $F_e\colon \modu S(n,d)\to \modu k\mathfrak{S}_d, M\mapsto eM$, sometimes called the \emphbf{Schur functor}, is an equivalence for $p=0$ or $p>d$. In this case, these categories are semi-simple.
\end{theorem}

In general, $F_e$ cannot be an equivalence, since $\mathfrak{S}_d$ is self-injective, hence of infinite global dimension if it is not semisimple, while $S(n,d)$ is always of finite global dimension. Recently, Hemmer and Nakano proved that in almost all cases there is an equivalence between two subcategories of these two algebras given by the Schur functor. The subcategory of the category of modules over the Schur algebra is similar to the subcategory of modules with a Verma flag introduced in the previous subsection. To define the analogues of the Verma modules for the Schur algebra we first introduce an analogue of the universal enveloping algebra of the Borel subalgebra, called the Borel Schur algebra. For more information on Borel Schur algebras, see e.g. \cite{Green-Oncertainsubalgebras-1990, Woodcock-BorelSchuralgebras-1994}

\begin{definition}
The algebra $S^+(n,d)$, defined as the subalgebra of $S(n,d)$ given by the images of the upper-triangular matrices under the map $k\GL(n,d)\to \End_{k\mathfrak{S}_d}(V^{\otimes d})$ is called the \emphbf{Borel Schur algebra}.
\end{definition}

The following is an analogue of the Borel--Weil theorem proved by Green, see \cite[(1.3)]{Green-Oncertainsubalgebras-1990}.

\begin{theorem}
Let $S(n,d)$ be a Schur algebra, $S^+(n,d)$ its Borel Schur subalgebra. Then, the following statements hold:
\begin{enumerate}[(i)]
\item The simple modules for $S^+(n,d)$ are indexed by the set of unordered partitions of $d$ into at most $n$ parts. All of them are $1$-dimensional, call them $k_\lambda$ for $\lambda$ an unordered partition.
\item The induced module $S(n,d)\otimes_{S^+(n,d)} k_\lambda$ is non-zero if and only if $\lambda$ is an ordered partition.
\end{enumerate}
\end{theorem}

For an ordered partition call $\Delta(\lambda):=S(n,d)\otimes_{S^+(n,d)} k_\lambda$ the \emphbf{Weyl module} corresponding to $\lambda$. The image of $\Delta(\lambda)$ under the Schur functor is called the \emphbf{Specht module} associated to $\lambda$.

\begin{theorem}[{\cite[Theorem 3.8.1]{HemmerNakano-Spechtfiltrations-2004}}]
Let $k$ be a field of characteristic $p\neq 2,3$. Let $S(n,d)$ be a Schur algebra. Denote by $\mathcal{F}(\Delta)\subset \modu S(n,d)$ the full subcategory of all modules having a filtration by Weyl modules and by $\mathcal{F}(e\Delta)\subset \modu \mathfrak{S}_d$ the full subcategory of all modules having a filtration by Specht modules. Then the Schur functor restricts to an equivalence between $\mathcal{F}(\Delta)$ and $\mathcal{F}(e\Delta)$.
\end{theorem}

In particular, more detailed information on the representation theory of the Schur algebra, in many cases, also tells us something about the representation theory of the symmetric group.

Returning to the full module category, also the Schur algebras have been classified according to their representation type. The following theorem summarises the work of many people. The first instance was achieved by Xi in \cite{Xi-Onrepresentationtypes-1992} under the constraint $n\geq d$  who classified these Schur algebras as to whether they are representation-finite or not. In this case the representation theory of $S(n,d)$ is more regular and the module category is equivalent to the category of strict polynomial functors, see \cite[Theorem 3.2]{FriedlanderSuslin-Cohomologyoffinitegroupschemes-1997}. Erdmann removed this constraint in \cite{Erdmann-Schuralgebras-1993}. The representation type of all blocks was established by Donkin and Reiten in \cite{DonkinReiten-OnSchuralgebras-1994} by classifying all quasi-hereditary algebras of finite representation type having a duality. Finally, the classification of tame Schur algebras is due to Doty, Erdmann, Martin, and Nakano in \cite{DotyErdmannMartinNakano-Representationtype-1999}. The labelling of the tame algebras follows Erdmann in \cite{Erdmann-Schuralgebras-1993}.

\begin{theorem}\label{representationtypeSchur}
Let $k$ be of characteristic $p\geq 0$.
\begin{enumerate}[(i)]
\item The Schur algebra $S(n,d)$ has finite representation type if and only if one of the following occurs:
\begin{enumerate}[(a)]
\item $p\geq 2, n\geq 3, d<2p$,
\item $p\geq 2$, $n=2$, $d<p^2$,
\item $p=2$, $n=2, d=5,7$.
\end{enumerate}
Furthermore, each representation-finite block of $S(n,d)$ is Morita equivalent to $(\mathcal{A}_n)$, the path algebra of $\begin{tikzcd}1\arrow[yshift=0.5ex]{r}{\alpha_1}&2\arrow[yshift=-0.5ex]{l}{\beta_1}\arrow[yshift=0.5ex]{r}{\alpha_2}&\dots\arrow[yshift=-0.5ex]{l}{\beta_2}\arrow[yshift=0.5ex]{r}{\alpha_{m-1}}&m\arrow[yshift=-0.5ex]{l}{\beta_{m-1}}\end{tikzcd}$ with relations $\alpha_{i-1}\beta_{i-1}-\beta_i\alpha_i$, $\alpha_{i}\alpha_{i-1}$, $\beta_{i-1}\beta_i$ for $i=2,\dots, m-1$ and $\alpha_{m-1}\beta_{m-1}$, where $m$ is the number of simple modules in that block.
\item The Schur algebra $S(n,d)$ has tame representation type if and only if one of the following occurs:
\begin{enumerate}[(a)]
\item $p=2, n=2, d=4,9$,
\item $p=3, n=3, d=7$,
\item $p=3, n=3, d=8$,
\item $p=3, n=2, d=9,10,11$.
\end{enumerate}
The basic algebras corresponding to the non-semisimple blocks in these cases are all isomorphic to one of the cases in the following list of path algebras with relations (in the same ordering):
\begin{enumerate}[(a)]
\item[$(\mathcal{D}_3)$] $\begin{tikzcd}3\arrow[yshift=0.5ex]{r}{\alpha_1}&1\arrow[yshift=0.5ex]{r}{\beta_2}\arrow[yshift=-0.5ex]{l}{\beta_1}&2\arrow[yshift=-0.5ex]{l}{\alpha_2}\end{tikzcd}$ with relations $\beta_1\alpha_1, \beta_2\alpha_2, \alpha_2\beta_2\alpha_1, \beta_1\alpha_2\beta_2$,
\item[$(\mathcal{R}_4)$] $\begin{tikzcd}4\arrow[yshift=0.5ex]{r}{\alpha_3}&3\arrow[yshift=0.5ex]{r}{\alpha_2}\arrow[yshift=-0.5ex]{l}{\beta_3}&2\arrow[yshift=0.5ex]{r}{\alpha_1}\arrow[yshift=-0.5ex]{l}{\beta_2}&1\arrow[yshift=-0.5ex]{l}{\beta_1}\end{tikzcd}$ with relations $\beta_3\alpha_3$, $\alpha_2\alpha_3$, $\beta_3\beta_2$, $\alpha_3\beta_3=\beta_2\alpha_2$, $\alpha_2\beta_2=\beta_1\alpha_1$,
\item[$(\mathcal{H}_4)$] $\begin{tikzcd}{}&3\arrow[xshift=0.5ex]{d}{\alpha_2}\\
4\arrow[yshift=0.5ex]{r}{\alpha_1}&2\arrow[yshift=-0.5ex]{l}{\beta_1}\arrow[xshift=-0.5ex]{u}{\beta_2}\arrow[yshift=0.5ex]{r}{\beta_3}&1\arrow[yshift=-0.5ex]{l}{\alpha_3}\end{tikzcd}$ with relations $\beta_1\alpha_1$, $\beta_1\alpha_2$, $\beta_1\alpha_3$, $\beta_2\alpha_1$, $\beta_2\alpha_2$, $\beta_3\alpha_1$, $\alpha_3\beta_3=\alpha_1\beta_1+\alpha_2\beta_2$,
\item[$(\mathcal{D}_4)$] The same quiver as in (c) but with relations $\beta_1\alpha_1$, $\beta_2\alpha_2$, $\beta_1\alpha_3$, $\beta_2\alpha_3$, $\beta_3\alpha_1$, $\beta_3\alpha_2$, $\alpha_2\beta_2=\alpha_3\beta_3$.
\end{enumerate}
\end{enumerate}
\end{theorem}

Note that there is no theorem classifying all tame blocks of Schur algebras. The author does not know of any other example of a tame block of a Schur algebra.

Unlike the analogous case of blocks of category $\mathcal{O}$, the representation type of the subcategory of modules filtered by Weyl modules $\mathcal{F}(\Delta)\subseteq S(n,d)$ is not known in general. Let us state here two theorems in this direction. The first theorem gives a sufficient criterion for $\mathcal{F}(\Delta)\subseteq \modu S(n,d)$ not to be of finite representation type:

\begin{theorem}[{\cite[Corollary 6.11]{CoxErdmann-OnExt2-2000}}]
Let $k$ be a field of characteristic $p\geq 0$. Let $S(n,d)$ be a Schur algebra. Then $\mathcal{F}(\Delta)\subseteq \modu S(n,d)$ is of infinite representation type if one of the following cases holds:
\begin{enumerate}[(a)]
\item $p>2$ and $d\geq 2p^2+p-2$, or
\item $p=2$ and $d\geq \begin{cases}8&\text{if $d$ is even}\\17&\text{if $d$ is odd}\end{cases}$
\end{enumerate}
\end{theorem}

A necessary and sufficient criterion is only available in the case $n=2$.

\begin{theorem}[{\cite{ErdmannMadsenMiemietz-OnAuslanderReitentranslates-2010}}]
Let $k$ be a field of characteristic $p\geq 0$. Let $A$ be a block of the Schur algebra $S(2,d)$ over $k$ with $m$ simple modules. Then the corresponding $\mathcal{F}(\Delta)$ is representation-finite if and only if one of the following cases occurs:
\begin{enumerate}[(a)]
\item $p=2$ and $m\leq 4$,
\item $p\leq 7$, $p$ odd, and $m\leq p+3$,
\item $p\geq 11$ and $m\leq p+2$.
\end{enumerate} 
\end{theorem}

The proof uses quite sophisticated methods from Auslander--Reiten theory of finite dimensional algebras.\\
We will not spend much time on the definition of the $q$-Schur algebra. It is defined analogously to the classical Schur algebra, using all the time the quantised version. The general linear group $\GL(n)$ is replaced by $U_q(\mathfrak{gl}_n)$, the symmetric group $\mathfrak{S}_d$ by the Hecke algebra $\mathcal{H}_d$, and also the Weyl modules have quantum analogues. For details, see \cite{Donkin-TheqSchuralgebra-1998}.
Regarding representation type there is the following classification due to Erdmann and Nakano \cite{ErdmannNakano-Representationtype-2001}. Again, the representation-finite case for $n\geq d$ was established earlier by Xi in \cite{Xi-Onrepresentationtypesq-1993}.

\begin{theorem}
Let $k$ be a field of characteristic $p\geq 0$. If $q$ is not a root of unity, then $S_q(n,d)$ is semisimple. In particular, it is representation-finite. If on the other hand $q\in k$ is a primitive $\ell$-th root of unity, then there is the following distinction of representation types:
\begin{enumerate}[(i)]
\item The $q$-Schur algebra is of finite representation type if and only if one of the following cases occurs:
\begin{enumerate}[(a)]
\item $n\geq 3$ and $d<2\ell$,
\item $n=2$, $p=0$, 
\item $n=2$, $p\neq 0$, $\ell\geq 3$ and $d<\ell p$,
\item $n=2$, $p\geq 3$, $\ell=2$ and $d$ even with $d<2p$, or $d$ odd with $d<2p^2+1$.
\end{enumerate}
\item The $q$-Schur algebra is of tame representation type if and only if one of the following cases occurs:
\begin{enumerate}[(a)]
\item $n = 3$, $\ell = 3$, $p \neq 2$, and $d = 7,8$,
\item $n = 3$, $\ell = 2$, and $d = 4,5$,
\item $n = 4$, $\ell = 2$, and $d = 5$,
\item $n = 2$, $\ell \geq 3$, $p = 2$ or $p = 3$, and $p\ell \leq d < (p + 1)\ell$
\item $n = 2$, $\ell = 2$, $p = 3$ and $d \in \{6, 19, 21, 23\}$
\end{enumerate}
In each of the cases, the tame block occurring is provided in the following list. In some cases, it occurs twice or some additional representation-finite block occurs (see the cited reference for details):
\begin{enumerate}[(a)]
\item $(\mathcal{R}_4)$ for $d=7$, or $(\mathcal{H}_4)$ for $d=8$,
\item $(\mathcal{R}_4)$ for $d=4$, or $(\mathcal{H}_4)$ for $d=5$,
\item $(\mathcal{H}_4)$,
\item $(\mathcal{D}_3)$ for $p=2$, or $(\mathcal{D}_4)$ for $p=3$,
\item $(\mathcal{D}_4)$
\end{enumerate}
\end{enumerate}
\end{theorem}

This result tells us that everything we prove for the tame Schur algebras via the basic algebras of their blocks will also be valid for the $q$-Schur algebras.

\subsection{Quasi-hereditary algebras and exact Borel subalgebras}

Summarising the common features between blocks of BGG category $\mathcal{O}$ and Schur algebras (and other categories, most notably the category of $G_rT$, for $G_r$ the $r$-th Frobenius kernel of a reductive algebraic group $G$ with torus $T$) one arrives at the notion of a highest weight category, or if one restricts to finitely many simple modules, at the notion of a quasi-hereditary algebra. There are many equivalent notions around to define it. We stick to the one closest to exceptional collections in triangulated categories (which often arise in algebraic geometry). For other possibilities, see the three surveys mentioned in the beginning of this section as well as the articles by Kalck \cite{Kalck-Derivedcategories-2016} and Krause \cite{Krause-Highestweightcategories-2016} in this volume. 

\begin{definition}
A finite dimensional algebra $A$ with $n$ simple modules (up to isomorphism) is called \emphbf{quasi-hereditary} if there exist indecomposable modules $\Delta(1),\dots,\Delta(n)$ with the following properties:
\begin{enumerate}[{(QH}1{)}]
\item $\End_A(\Delta(i))\cong k$,
\item $\Hom_A(\Delta(i),\Delta(j))\neq 0\Rightarrow i\leq j$,
\item $\Ext^1_A(\Delta(i),\Delta(j))\neq 0\Rightarrow i<j$,
\item $A\in \mathcal{F}(\Delta)$.
\end{enumerate}
\end{definition}

\begin{remark}
\begin{enumerate}[(i)]
\item The modules $\Delta(1),\dots, \Delta(n)$ are called \emphbf{standard modules}.
\item Sometimes it is convenient to work with a partial order on $\{1,\dots,n\}$ instead of a total order, but the two definitions are equivalent.
\item Let $M(i)=\displaystyle\sum_{\substack{j>i\\f\in \Hom(P_j,P_i)}} \image(f)$. Then the standard modules can be defined as $\Delta(i):=P(i)/M(i)$. One equivalent definition for an algebra to be quasi-hereditary one can easily check in examples is that the so-defined modules $\Delta(i)$ filter the indecomposable projectives for the algebra.
\end{enumerate}
\end{remark}

Motivated by the example of $U(\mathfrak{b})\subseteq U(\mathfrak{g})$, Koenig introduced the notion of an exact Borel subalgebra of a quasi-hereditary algebra \cite{Koenig-ExactBorelsubalgebras-1995}. 

\begin{definition}
Let $A$ be a quasi-hereditary algebra. A subalgebra $B\hookrightarrow A$ is called an \emphbf{exact Borel subalgebra} if the number of simples for $B$ and $A$ coincides and the following properties hold:
\begin{enumerate}[{(B}1{)}]
\item The algebra $B$ is \emphbf{directed}, i.e. the Gabriel quiver of $B$ is directed, i.e. whenever there is an arrow $i\to j$ in $Q_B$, then $i\leq j$.
\item The induction functor $A\otimes_B - $ is exact.
\item The standard modules for $A$ can be obtained as $\Delta_A(i)=A\otimes_B L_B(i)$, where $L_B(i)$ are the simple modules for $B$.
\end{enumerate}
\end{definition}

Note that the Borel Schur algebras defined before aren't examples of exact Borel subalgebras, as the number of simple modules for $S^+(n,d)$ and $S(n,d)$ does not coincide. A paper establishing common grounds for these definitions is \cite{ParshallScottWang-Borelsubalgebrasredux-2000}.

Quasi-hereditary algebras in general do not have exact Borel subalgebras. An example was already given in \cite[Example 2.3]{Koenig-ExactBorelsubalgebras-1995}, see also \cite[Appendix A.3]{KoenigKulshammerOvsienko-Quasihereditaryalgebras-2014}. However, Koenig was able to establish the phenomenon that the quasi-hereditary algebras coming from blocks of BGG category $\mathcal{O}$ seem to remember that they came from $U(\mathfrak{g})$ and have exact Borel subalgebras.

\begin{theorem}[{\cite[Theorem D]{Koenig-ExactBorelsubalgebras-1995}}]
Let $\mathfrak{g}$ be a finite dimensional semismiple complex Lie algebra. Let $\mathcal{O}_\chi$ be a block of BGG category $\mathcal{O}$. Then, there exists a quasi-hereditary algebra $A_\chi'$ with $\modu A_\chi'\cong \mathcal{O}_\chi$ having an exact Borel subalgebra.
\end{theorem}

The question whether for general quasi-hereditary algebras there exists a Morita equivalent algebra having an exact Borel subalgebra was open for a long time, even for the class of Schur algebras. The folllowing theorem by Koenig, Ovsienko, and the author setteled the question affirmatively:

\begin{theorem}
For every quasi-hereditary algebra $A$ there exists a Morita equivalent algebra $R$ such that $R$ has an exact Borel subalgebra $B$.
\end{theorem}

In fact, the authors proved a stronger statement. We will describe some of the techniques involved in the next section. It is not known whether in the case of blocks of BGG category $\mathcal{O}$ the exact Borel subalgebras defined by Koenig and by Koenig, Ovsienko, and the author coincide.

\section{Bocses}

Bocses were introduced in 1979 by Roiter \cite{Roiter-Matrixproblems-1979}. The term bocs is an acronym for \underline{b}imodule \underline{o}ver \underline{c}ategory with \underline{c}oalgebra \underline{s}tructure. In this article, we do not need this generality and will stick to Burt--Butler's generality of using a bimodule over a finite dimensional algebra. In our setting, we can and will always assume the algebra $B$ to be basic. For such an  algebra fix an isomorphism $B\cong KQ/I$ for a quiver $Q$ and an admissible ideal $I$, and regard it as a category with objects the vertices of the quiver and morphisms the equivalence classes of linear combinations of paths in $Q$ modulo the ideal $I$.
In other contexts, bocses are also called $B$-corings for $B$ an associative algebra.

\subsection{Definitions and first properties}

In this subsection, we introduce the notion of a bocs and its category of representations. For general theory on bocses, including parts of the reduction algorithm, the reader is referred to the book by Bautista, Salmer\'on, and Zuazua \cite{BautistaSalmeronZuazua-Differentialtensoralgebras-2009}. Nicely written approaches to other aspects of the theory are contained in the unpublished manuscript by Burt \cite{Burt-Almostsplitsequences-2005} and the survey by Crawley-Boevey \cite{CrawleyBoevey-Matrixproblems-1990}.

\begin{definition}
A \emphbf{bocs} is a pair $\mathfrak{B}=(B,W)$ consisting of a finite dimensional algebra $B$ and a $B$-$B$-bimodule $W$ which has a $B$-coalgebra structure, i.e. there exists a $B$-bilinear comultiplication $\mu\colon W\to W\otimes_B W$ and a $B$-bilinear counit $\varepsilon\colon W\to B$ such that the following diagrams commute:
\[\begin{tikzcd}
W\arrow{d}{\mu}\arrow{r}{\mu} &W\otimes_B W\arrow{d}{1\otimes_B \mu}\\
W\otimes_B W\arrow{r}{\mu\otimes_B 1} &W\otimes_B W\otimes_B W
\end{tikzcd}
\begin{tikzcd}
B\otimes_B W\arrow{dr}[swap]{\can}&W\otimes_B W\arrow{r}{1\otimes_B \varepsilon}\arrow{l}[swap]{\varepsilon\otimes_B 1} &W\otimes_B B\arrow{dl}{\can}\\
&W\arrow{u}{\mu}
\end{tikzcd}
\]
\end{definition}

\begin{remark}
\begin{enumerate}[(i)]
\item The slightly unorthodox use of the letter $\mu$ for the comultiplication is due to historic reasons. For us it is fortunate since the more standard letter $\Delta$ is already in use for standard modules.
\item In addition to the conditions mentioned in the definition we will also assume that $B$ is basic and that $W$ is finite dimensional.
\end{enumerate}
\end{remark}

\begin{example}
The easiest possible example is the \emphbf{regular bocs}, where $W=B$, $\mu\colon B\to B\otimes_B B$ is the canonical isomorphism, and $\varepsilon=1_B$.
\end{example}

\begin{definition}
Let $\mathfrak{B}=(B,W)$ be a bocs. Then, the category of \emphbf{modules} over the bocs $\mathfrak{B}$, $\Mod \mathfrak{B}$ (resp. finite dimensional modules $\modu \mathfrak{B}$) is defined as follows:
\begin{description}
\item[objects:] (finite dimensional) $B$-modules
\item[morphisms:] $\Hom_\mathfrak{B}(M,N):=\Hom_{B\otimes B^{op}}(W,\Hom_k(M,N))$
\item[composition:] The composition of two morphisms $f\colon L\to M$ and $g\colon M\to N$ is given by the following composition of $B$-bilinear maps:
\[
\begin{tikzcd}
W\arrow{r}{\mu} &W\otimes_B W\arrow{r}{g\otimes f} &\Hom_k(M,N)\otimes_B \Hom_k(L,M)\\
{}\arrow{r}{\comp}&\Hom_k(L,M)
\end{tikzcd}\]
where $\comp$ denotes the usual composition of functions.
\item[units:] The unit morphism $1_M\in \Hom_{\mathfrak{B}}(M,M)$ is given by the composition of the following maps
\[
\begin{tikzcd}
W\arrow{r}{\varepsilon} &B\arrow{r}{\lambda} &\Hom_k(M,M)
\end{tikzcd}
\]
where $\lambda$ is the function mapping an element $b\in B$ to left multiplication with $b$.
\end{description}
\end{definition}

\begin{remark}\label{bocssecondthirddefinition}
\begin{enumerate}[(i)]
\item This is not the original definition of modules over a bocs, but it is an equivalent one. Using a standard adjunction
\begin{align*}
\Hom_{B\otimes B^{op}}(W,\Hom_k(M,N))&\cong \Hom_B(W\otimes_B M,N)\\
&\cong \Hom_B(M,\Hom_B(W,N))
\end{align*}
one gets back two of the more standard definitions.\\
Using structure transport one sees that the multiplication and units for the second one are given as follows:
\begin{description}
\item[composition:] The composition of $g\colon W\otimes_B M\to N$ and $f\colon W\otimes_B L\to M$ is given by composing the following functions:
\[
\begin{tikzcd}
W\otimes_B L\arrow{r}{\mu\otimes 1} &W\otimes_B W\otimes_B L\arrow{r}{1\otimes_B f} &W\otimes_B M\arrow{r}{g} &N
\end{tikzcd}
\]
\item[units:] $1_M$ is given by $\begin{tikzcd}W\otimes_B M\arrow{r}{\varepsilon\otimes 1} &B\otimes_B M\arrow{r}{\can} &M\end{tikzcd}$, where $\can$ denotes the canonical isomorphism.
\end{description}
For the third definition, structure transport gives the following:
\begin{description}
\item[composition:] The composition of two morphisms $g\colon M\to \Hom_B(W,N)$ and $f\colon L\to \Hom_B(W,M)$ is given by composing the following functions:
\[\begin{tikzcd}
L\arrow{r}{f} &\Hom_B(W,M)\arrow{r}{\Hom_B(W,g)} &\Hom_B(W,\Hom_B(W,N))\\
{}\arrow{r}{\adj} &\Hom_B(W\otimes_B W,N)\arrow{r}{\Hom(\mu,N)} &\Hom_B(W,N),
\end{tikzcd}\]
where $\adj$ denotes the canonical $\Hom$-tensor adjunction homomorphism.
\item[units:] $1_M$ is given by 
\[\begin{tikzcd}[column sep=10ex]M\arrow{r}{\can} &\Hom_B(B,M)\arrow{r}{\Hom_B(\varepsilon,M)} &\Hom_B(W,M)\end{tikzcd},\]
where $\can$ denotes the canonical isomorphism.
\end{description}
The second and the third definition do in fact fit into a more general framework called Kleisli categories of monads and comonads. We will discuss this in the next section. Although it might not seem so at first sight, the first definition is actually the closest one to the representation theory of quivers. Choosing generators of $W$ one can describe morphisms in a very similar way to representations of quivers. We will see one such description in Subsection \ref{bocsesviadifferentialbiquivers}.
\item In the case of the regular bocs, one gets $\modu \mathfrak{B}\cong \modu B$.
\end{enumerate}
\end{remark}

\subsection{Monads and comonads}

In this subsection, monads, comonads, and their Kleisli categories are discussed which form a general framework for the equivalence of the second and third definition of modules over bocses. This possibility was first observed by Bautista, Colavita, and Salmer\'on in \cite{BautistaColavitaSalmeron-Onadjointfunctors-1981}. The content of this subsection is taken quite literally from \cite{Kleiner-Inducedmodules-1981}. For a more recent paper on the subject, see \cite{BohmBrzezinskiWisbauer-Monadsandcomonads-2009}.

\begin{definition}
Let $\mathcal{C}$ be a category. 
\begin{enumerate}[(i)]
\item A \emphbf{monad} on $\mathcal{C}$ consists of an endofunctor $T\colon \mathcal{C}\to \mathcal{C}$ together with two natural transformations $e\colon 1_\mathcal{C}\to T$ and $m\colon T\circ T\to T$ such that the following diagrams commute:
\[
\begin{tikzcd}
T^3\arrow{r}{Tm}\arrow{d}{mT} &T^2\arrow{d}{m}\\
T^2\arrow{r}{m}&T
\end{tikzcd}\qquad
\begin{tikzcd}
T\arrow{r}{Te}\arrow[equals]{rd} &T^2\arrow{d}{m}&T\arrow{l}[swap]{eT}\arrow[equals]{ld}\\
&T
\end{tikzcd}
\] 
\item A \emphbf{comonad} on $\mathcal{C}$ consists of an endofunctor $U\colon \mathcal{C}\to \mathcal{C}$ together with two natural transformations $f\colon U\to 1_\mathcal{C}$ and $c\colon U\to U\circ U$ such that the following diagrams commute:
\[
\begin{tikzcd}
U\arrow{r}{c}\arrow{d}{c} &U^2\arrow{d}{Uc}\\
U^2\arrow{r}{cU} &U^3
\end{tikzcd}\qquad 
\begin{tikzcd}
U&U^2\arrow{r}{Uf}\arrow{l}[swap]{fU}&U\\
&U\arrow{u}{c}\arrow[equals]{ru}\arrow[equals]{lu}
\end{tikzcd}
\]
\end{enumerate}
\end{definition}

Although there are many other examples, e.g. coming from logic, in this article, we will only consider the following examples coming from bocses.

\begin{example}
\begin{enumerate}[(i)]
\item A bocs $(B,W)$ induces a monad $T=\Hom_B(W,-)$ on $\modu B$ where $e\colon 1_{\modu B}\to T$ is given by the canonical isomorphism $1_{\modu B}\to \Hom_B(B,-)$ followed by the map $\Hom_B(\varepsilon,-)$ and $m\colon T^2\to T$ is given by the canonical isomorphism $\Hom_B(W,\Hom_B(W,-))\cong \Hom_B(W\otimes_B W,-)$ followed by $\Hom_B(\mu,-)$.
\item A bocs $(B,W)$ induces a comonad $U=W\otimes_B -$ on $\modu B$ where $f\colon U\to 1_{\modu B}$ is given by the composition of $\varepsilon\otimes_B -$ with the canonical isomorphism and $c\colon U\to U\circ U$ is given by $\mu\otimes_B -$.
\end{enumerate}
\end{example}

Associated to a monad, there is a category, which in the case of the monad (resp. comonad) associated to a bocs turns out to be the module category of the bocs.

\begin{definition}
Let $\mathcal{C}$ be a category.
\begin{enumerate}[(i)]
\item Let $T$ be a monad on $\mathcal{C}$. Then, the \emphbf{Kleisli category} $\Kleisli T$ of $T$ is the category with
\begin{description}
\item[objects:] The objects are the same as the objects of $\mathcal{C}$.
\item[morphisms:] Given two objects, $X,Y\in \Kleisli T$, 
\[\Hom_T(X,Y):=\Hom_{\mathcal{C}}(X,TY).\]
\item[composition:] Given three objects $X,Y,Z\in \Kleisli T$, the composition map
\[\Hom_{\mathcal{C}}(Y,TZ)\times \Hom_{\mathcal{C}}(X,TY)\to \Hom_{\mathcal{C}}(X,TZ)\]
is given by the composition of the following maps 
\[\begin{tikzcd}X\arrow{r}{f}&TY\arrow{r}{Tg}&T^2 Z\arrow{r}{m_Z}&TZ\end{tikzcd}.\]
\item[unit:] The unit morphism in $\Hom_{\mathcal{C}}(X,TX)$ is given by $e_X$.
\end{description}
\item Dually, given a comonad $U$ on $\mathcal{C}$, the corresponding \emphbf{coKleisli category} $\Kleisli U$ of $U$ is the category with
\begin{description}
\item[objects:] The objects are the same as the objects of $\mathcal{C}$.
\item[morphisms:] Given $X,Y\in \Kleisli U$, $\Hom_U(X,Y):=\Hom_{\mathcal{C}}(UX,Y)$.
\item[composition:] Given three objects, $X,Y,Z \in \Kleisli U$, the composition map
\[\Hom_{\mathcal{C}}(UY,Z)\times \Hom_{\mathcal{C}}(UX,Y)\to \Hom_{\mathcal{C}}(UX,Z)\]
is given by the composition of the following maps
\[\begin{tikzcd}UX\arrow{r}{c_X} &U^2X\arrow{r}{Uf} &UY\arrow{r}{g} &Z.\end{tikzcd}\]
\item[unit:] The unit morphism in $\Hom_{\mathcal{C}}(UX,X)$ is given by $f_X$.
\end{description}
\end{enumerate}
\end{definition}

\begin{remark}
\begin{enumerate}[(i)]
\item The associativity of composition and the unitality follow from the corresponding properties of the monad, respectively comonad.
\item If $\mathfrak{B}=(B,W)$ is a bocs then, by definition and by the equivalent descriptions of $\modu\mathfrak{B}$ given in the previous section, $\modu \mathfrak{B}\cong \Kleisli U_{\mathfrak{B}}\cong \Kleisli T_{\mathfrak{B}}$. This is a special case of a general phenomenon observed by Kleiner in \cite{Kleiner-Inducedmodules-1981} that will be recalled next.
\end{enumerate}
\end{remark}

\begin{definition}
Let $\mathcal{C}$ be a category. Let $T$ (resp. $U$) be a monad (resp. comonad) on $\mathcal{C}$. Then $T$ is said to be \emphbf{right adjoint} to $U$ \emphbf{(in the monadic sense)} if there exists an adjunction $\alpha_{X,Y}\colon \begin{tikzcd}\Hom_{\mathcal{C}}(UX,Y)\arrow{r}{\sim}  &\Hom_{\mathcal{C}}(X,TY)\end{tikzcd}$ such that the following two diagrams commute for all $X,Y\in \mathcal{C}$:
\[
\begin{tikzcd}[column sep=10ex]
\Hom_{\mathcal{C}}(UX,Y)\arrow{r}{\alpha_{X,Y}} &\Hom_{\mathcal{C}}(X,TY)\\
\Hom_{\mathcal{C}}(U^2X,Y)\arrow{r}{\alpha_{X,TY}\circ \alpha_{UX,Y}}\arrow{u}{\Hom_{\mathcal{C}}(c_X,Y)} &\Hom_{\mathcal{C}}(X,T^2Y)\arrow{u}[swap]{\Hom_{\mathcal{C}}(X,m_Y)}
\end{tikzcd}
\]
and
\[
\begin{tikzcd}
\Hom_{\mathcal{C}}(UX,Y)\arrow{rr}{\alpha_{X,Y}} &&\Hom_{\mathcal{C}}(X,TY)\\
&\Hom_{\mathcal{C}}(X,Y)\arrow{lu}{\Hom_{\mathcal{C}}(f_X,Y)}\arrow{ru}[swap]{\Hom_{\mathcal{C}}(X,e_Y)}
\end{tikzcd}
\]
\end{definition}

\begin{theorem}
Let $\mathcal{C}$ be a category. Let $T$ be a monad on $\mathcal{C}$ which is right adjoint to a comonad $U$ on $\mathcal{C}$ in the monadic sense. Then, the Kleisli category of $T$ is equivalent to the coKleisli category of $U$.
\end{theorem}

Next, we will recall the notion of the Eilenberg--Moore category of a monad, which gives rise to two further equivalent definitions of the category of modules over a bocs.

\begin{definition}
Let $\mathcal{C}$ be a category.
\begin{enumerate}[(i)]
\item Let $T$ be a monad on $\mathcal{C}$. Then, the \emphbf{Eilenberg--Moore category} $\EM(T)$ is defined to be the category with:
\begin{description}
\item[objects:] pairs $(X,h)$, where $X\in \mathcal{C}$ and $h\in \Hom_{\mathcal{C}}(TX,X)$, such that the following diagrams commute:
\[\begin{tikzcd}
T^2X\arrow{r}{m_X}\arrow{d}{Th} &TX\arrow{d}{h}\\
TX\arrow{r}{h}&X
\end{tikzcd}\qquad
\begin{tikzcd}
TX\arrow{r}{h}&X\arrow[equals]{ld}\\
X\arrow{u}{e_X}
\end{tikzcd}\]
\item[morphisms:] $\Hom_{\EM(T)}((X,h),(X',h'))$ is the set given by those morphisms $f\in \Hom_{\mathcal{C}}(X,X')$ such that the following diagram commutes:
\[
\begin{tikzcd}
X\arrow{r}{f}&X'\\
TX\arrow{u}{h}\arrow{r}{Tf}&TX'\arrow{u}{h'}
\end{tikzcd}\]
\item[composition:] Composition is given by the composition in $\mathcal{C}$.
\item[unit:] The unit is given by the identity map $1_X\in \Hom_{\mathcal{C}}(X,X)$.
\end{description} 
\item Dually, let $U$ be a comonad on $\mathcal{C}$. Then, the \emphbf{co-Eilenberg--Moore category} $\EM(U)$ is defined to be the category with:
\begin{description}
\item[objects:] pairs $(Y,g)$, where $Y\in \mathcal{C}$ and $g\in \Hom_{\mathcal{C}}(Y,UY)$, such that the following diagrams commute:
\[\begin{tikzcd}
Y\arrow{r}{g}\arrow{d}{g}&UY\arrow{d}{c_Y}\\
UY\arrow{r}{Ug}&U^2 Y
\end{tikzcd}\qquad
\begin{tikzcd}
Y\arrow{d}{g}\arrow[equals]{rd}\\
UY\arrow{r}{f_Y} &Y
\end{tikzcd}\]
\item[morphisms:] $\Hom_{\EM(U)}((Y,g),(Y',g'))$ is the set given by those morphisms $g\in \Hom_{\mathcal{C}}(Y,Y')$ such that the following diagram commutes:
\[\begin{tikzcd}
Y\arrow{d}{g}\arrow{r}{f} &Y'\arrow{d}{g'}\\
UY\arrow{r}{Uf} &UY'
\end{tikzcd}\]
\item[composition:] Composition is given by the composition in $\mathcal{C}$.
\item[unit:] The unit is given by the idenity map $1_Y\in \Hom_{\mathcal{C}}(Y,Y)$.
\end{description}
\item The \emphbf{free Eilenberg--Moore category} is the full subcategory of the Eilen\-berg--Moore category given by the objects of the form $(TX,m_X)$.
\item The \emphbf{cofree Eilenberg--Moore category} is the full subcategory of the co-Eilenberg--Moore category given by the objects of the form $(UY,c_Y)$.
\end{enumerate}
\end{definition}

\begin{example}
Let $\mathfrak{B}=(B,W)$ be a bocs.
\begin{enumerate}[(i)]
\item For the associated monad $T$, the free Eilenberg--Moore category is equal to the category of coinduced $W$-contramodules, i.e. the full subcategory of all contramodules isomorphic to $\Hom_B(W,M)$ where $M$ is a $B$-module and the contraaction is given by the composition of 
\[\begin{tikzcd}[column sep=1.6cm]\Hom_B(W,\Hom_B(W,M))\arrow{r}{\can} &\Hom_B(W\otimes_B W,M)\arrow{d}{\Hom_B(\mu,M)}\\
{} &\Hom_B(W,M).\end{tikzcd}\]
\item For the associated comonad $U$, the cofree Eilenberg--Moore category is equal to the category of induced $W$-comodules, i.e. the full subcategory of all comodules isomorphic to $W\otimes_B M$ for a $B$-module $M$ where the coaction is given by $\begin{tikzcd}W\otimes_B M\arrow{r}{\mu\otimes M} &W\otimes W\otimes M\end{tikzcd}$.
\end{enumerate}
\end{example}

\begin{lemma}
Let $\mathcal{C}$ be a category.
\begin{enumerate}[(i)]
\item Let $T$ be a monad on $\mathcal{C}$. Let $F\colon \Kleisli T\to \EM(T)$ be the functor given on objects by $X\mapsto TX$ and on morphisms by $f\colon X\to TY$ is mapped to the composition $\begin{tikzcd}TX\arrow{r}{Tf}&T^2Y\arrow{r}{m_Y} &TY\end{tikzcd}$. Then $F$ is fully faithful with essential image the free Eilenberg--Moore category.
\item Let $U$ be a comonad on $\mathcal{C}$. Let $G\colon \Kleisli U\to \EM(U)$ be the functor given on objects by $Y\mapsto UY$ and on morphisms by $g\colon UX\to Y$ is mapped to the composition $\begin{tikzcd}UX\arrow{r}{c_X} &U^2X\arrow{r}{Ug} &UY\end{tikzcd}$. Then $G$ is fully faithful with essential image the cofree co-Eilenberg--Moore category.
\end{enumerate}
\end{lemma}

\begin{example}
Let $\mathfrak{B}=(B,W)$ be a bocs. In this case, the above lemma states that the category of modules over the bocs is equivalent to the category of coinduced $W$-contramodules as well as to the category of induced $W$-comodules.
\end{example}

So far, we have seen five ways to describe the category of modules over a bocs: the bimodule definition, the Kleisli category of $\Hom_B(W,-)$ as well as the coKleisli category of $W\otimes_B -$ and, at the end of this subsection, the category of induced $W$-comodules as well as the category of coinduced $W$-contramodules. In the next subsection, using a duality, there will be two further descriptions.

\subsection{Burt--Butler theory of bocses}

If $C$ is a coalgebra over a field $k$, it is well-known that its dual $DC$ is an algebra over $k$. In 1975 Sweedler, in \cite{Sweedler-Thepredualtheorem-1975}, considered the analogous situation over a not-necessarily commutative ring $B$. He defined for a bocs $\mathfrak{B}=(B,W)$ two duals and showed that certain subcategories of modules and comodules are equivalent. This was further developed in \cite{Kleiner-Thedualring-1984}, and with the emphasis on Auslander--Reiten theory in \cite{BautistaKleiner-Almostsplitsequences-1990} and \cite{BurtButler-Almostsplitsequences-1991}. Other perspectives on this approach can for example be found in \cite{Wisbauer-Onthecategory-2002, PoinsotPorst-Thedualrings-2015}. We follow the approach by Burt and Butler, which is also contained in a manuscript by Burt \cite{Burt-Almostsplitsequences-2005}.

\begin{definition}
Let $\mathfrak{B}=(B,W)$ be a bocs.
\begin{enumerate}[(i)]
\item The \emphbf{right algebra} of $\mathfrak{B}$ is defined to be $R_{\mathfrak{B}}:=\End_{\mathfrak{B}}(B)^{op}$.
\item The \emphbf{left algebra} of $\mathfrak{B}$ is defined to be $L_{\mathfrak{B}}:=\End_{\mathfrak{B}^{op}}(B)$ where the endomorphism ring is computed in the category of right $\mathfrak{B}$-modules, i.e. in any of the definitions left modules are replaced by right modules.
\end{enumerate}
\end{definition}

\begin{remark}
Using the second description of the category of modules over a bocs $\mathfrak{B}=(B,W)$ it can easily be seen that the right algebra of the bocs $\mathfrak{B}$ is just the opposite ring of the $B$-dual algebra of the coalgebra $W$, i.e. $R\cong \Hom_B(W,B)$ with multiplication given by $f\cdot g=g(1\otimes_B f)\mu$, where we have omitted the canonical identification $W\otimes_B B\cong W$. A similar statement is true for $L$ considering right modules instead of left modules.
\end{remark}

Let us recall the following well-known definitions:

\begin{definition}
\begin{enumerate}[(i)]
\item A category $\mathcal{C}$ is called \emphbf{fully additive} (or idempotent closed) if for every $M\in \mathcal{C}$ and every idempotent morphism $e\colon M\to M$ there exists an object $N\in \mathcal{C}$ and morphisms $f\colon M\to N$ and $g\colon N\to M$ such that $e=gf$ and $1_N=fg$.
\item Let $A$ be an algebra. A module $T\in \modu A$ is called a \emphbf{(Miyashita) tilting module} if 
\begin{enumerate}[{(T}1{)}]
\item $T$ has finite projective dimension.
\item $\Ext^i(T,T)=0$ for $i>0$.
\item There exists a short exact sequence $0\to A\to T^{(1)}\to \dots\to T^{(m+1)}\to 0$ for some $m$ with $T^{(j)}\in \add T$, i.e. $T^{(j)}$ is a direct summand of a finite direct sum of copies of $T$.
\end{enumerate}
\item Let $\mathcal{C}$ be a Krull--Remak--Schmidt exact category. Let $X$ be an indecomposable object. A short exact sequence
\[0\to X\to Y\to Z\to 0\]
is called an \emphbf{Auslander--Reiten sequence} if it is not split and every map $Z'\to Z$ which is not a split epimorphism factors through $Y$. In this situation, we write $Z=\tau^{-1}X$ and $X=\tau Z$.
\item A Krull--Remak--Schmidt exact category \emphbf{admits Auslander--Reiten sequences} if for every non-projective object $Z$ there exists an Auslander--Reiten sequence ending in $Z$.
\end{enumerate}
\end{definition}

In order for $\modu \mathfrak{B}$ to admit Auslander--Reiten sequences the bocses need to satisfy the following quite restrictive condition.

\begin{definition}
A bocs $(B,W)$ is said to have \emphbf{projective kernel} if its counit $\varepsilon\colon W\to B$ is surjective and $\ker \varepsilon$ is a projective $B$-$B$-bimodule.
\end{definition}

The following portmanteau theorem summarizes Burt--Butler's theory of bocses.

\begin{theorem}
Let $\mathfrak{B}=(B,W)$ be a bocs with projective kernel and assume that $\modu \mathfrak{B}$ is fully additive. Then the following statements hold:
\begin{enumerate}[(i)]
\item The algebra $B$ is a subalgebra of $R$ as well as of $L$.
\item The algebras $R$ and $L$ are projective over $B$.
\item The induction functor $R\otimes_B -\colon \modu B\to \modu R$ (resp. the coinduction fucntor $\Hom_B(L,-)\colon \modu B\to \modu L$) is faithful, but in general neither full nor dense. Its essential image is extension-closed in $\modu R$.
\item The induction functor $R\otimes_B-\colon \modu B\to \modu R$ (resp. the coinduction functor $\Hom_B(L,-)\colon \modu B\to \modu L$) induces morphisms 
\[\Ext^n_B(M,N)\to \Ext^n_R(R\otimes_B M, R\otimes_B N)\]
(resp. $\Ext^n_B(M,N)\to \Ext^n_L(\Hom_B(L,M),\Hom_B(L,N))$) which are epimorphisms for $n=1$ and isomorphisms for $n\geq 2$.
\item The functor $R\otimes_B -\colon \modu \mathfrak{B}\to \Ind(B,R)$ is an equivalence, where $\Ind(B,R)$ is the category of induced modules, i.e. modules of the form $R\otimes_B M$ for a $B$-module $M$. Analogously, the functor $\Hom_B(L,-)\colon \modu \mathfrak{B}\to \CoInd(B,L)$ is an equivalence, where $\CoInd(B,L)$ is the category of coinduced modules, i.e. modules of the form $\Hom_B(L,M)$ for a $B$-module $M$.
\item The $B$-$B$-bimodule $DW$ is also an $R$-$L$-bimodule which is an $R$-tilting module with $L\cong \End_R(DW)^{op}$ and $R\cong \End_{L}(DW)$. 
\item Let $\modu \mathfrak{B}$ be equipped with the exact structure given by restricting the exact structure of $\modu R$ to $\Ind(B,R)\cong \modu \mathfrak{B}$. Then $\modu \mathfrak{B}$ admits Auslander--Reiten sequences.
\end{enumerate}
\end{theorem}

\begin{remark}
The exact structure of $\modu \mathfrak{B}$ can in some cases be described intrinsically in terms of the bocs $\mathfrak{B}$ without referring to its right algebra $R$. For details, see e.g. \cite[Section 9]{KoenigKulshammerOvsienko-Quasihereditaryalgebras-2014}.
\end{remark}

Up to the authors knowledge, the following result has not been stated explicitly elsewhere. It can easily be obtained from the foregoing theorem.

\begin{corollary}\label{exactBorelfinitetype}
Let $\mathfrak{B}=(B,W)$ be a bocs with projective kernel. If $B$ is of finite representation type, then $\modu \mathfrak{B}\subseteq \modu R$ is of finite representation type.
\end{corollary}

\subsection{Quasi-hereditary algebras via bocses}

This subsection gives the characterisation of quasi-hereditary algebras by bocses with directed biquiver following \cite{KoenigKulshammerOvsienko-Quasihereditaryalgebras-2014}.

\begin{definition}
Let $\mathfrak{B}=(B,W)$ be a bocs with projective kernel. Let $\overline{W}:=\ker\varepsilon$.
\begin{enumerate}[(i)]
\item The \emphbf{biquiver} of $\mathfrak{B}$ is the quiver with vertices the vertices of the quiver of $B$ and two kinds of arrows (solid or dashed). The solid arrows are given by the quiver of $B$. The dashed arrows are obtained as follows: If $\overline{W}\cong \bigoplus Be_j\otimes_k e_mB$, then for each direct summand  there is a dashed arrow $\begin{tikzcd}m\arrow[dashed]{r} &j\end{tikzcd}$.
\item A biquiver with vertex set indexed by $\{1,\dots, n\}$ is called \emphbf{directed} if whenever there is an arrow $i\to j$ (regardless of whether it is solid or dashed), then $i\leq j$.
\end{enumerate}
\end{definition}

We now come to the equivalent description of quasi-hereditary algebras.

\begin{theorem}[{\cite[Theorem 1.1]{KoenigKulshammerOvsienko-Quasihereditaryalgebras-2014}}]\label{KoenigKulshammerOvsienko}
An algebra $A$ is quasi-hereditary if and only if there is a bocs with projective kernel $(B,W)$ with directed biquiver such that $A$ is Morita equivalent to the right algebra of $(B,W)$. The same holds for the right algebra being replaced by the left algebra.
\end{theorem}

The portmanteau theorem can now be specialised to this situation. For simplicity we will only consider the case of the right algebra. An analogous theorem holds when considering the left algebra. Part (ii) and (iii) are originally due to Ringel, see \cite{Ringel-Thecategoryofmodules-1991}.

\begin{theorem}
Let $R$ be the right algebra of a bocs with directed biquiver $\mathfrak{B}=(B,W)$.
\begin{enumerate}[(i)]
\item Then $B$ is an Borel subalgebra of $R$ such that the induction functor $R\otimes_B-\colon \modu B\to \mathcal{F}(\Delta)$ is faithful and dense. It is full if and only if the standard modules are simple if and only if $\overline{W}=0$. Furthermore, this induction functor preserves $\Ext$-groups starting from $n\geq 2$.
\item There is a tilting module $DW\in \modu R$. This is called the \emphbf{characteristic tilting module}. (To be precise one should say up to multiplicity of direct summands.) The left algebra $L$ is called the \emphbf{Ringel dual} of $R$.
\item The category $\mathcal{F}(\Delta)$ admits Auslander--Reiten sequences.
\end{enumerate}
\end{theorem}

Since in Lie theory a Borel subalgebra $\mathfrak{b}\subset \mathfrak{g}$ is unique up to an inner automorphism of $\mathfrak{g}$, it is natural to ask for a sort of uniqueness of an exact Borel subalgebra or of a bocs $\mathfrak{B}$. In this direction, the following is a result of joint work with Vanessa Miemietz.

\begin{theorem}[{\cite{KulshammerMiemietz-Uniqueness-2016}}]
Let $\mathfrak{B}=(B,W)$ and $\mathfrak{B}'=(B',W')$ be two bocses with basic $B$ and Morita equivalent right algebras, such that $\Ext^1_B(L,L)\cong \Ext^1_A(\Delta,\Delta)$, where $L$ is the direct sum of all simple $B$-modules and $\Delta$ the direct sum of all standard modules for $A$. Then, the biquivers associated to $\mathfrak{B}$ and $\mathfrak{B}'$ coincide and $B\cong B'$.
\end{theorem}

The restriction that $\Ext^1_B(L,L)\cong \Ext^1_A(\Delta,\Delta)$ is not essential. It can always be achieved using the process of regularisation described in Proposition \ref{regularisation}.

\subsection{Curved differential graded categories and bocses}

In this short subsection, we will state the equivalence between bocses with surjective counit and curved differential graded algebras. This was first proved in \cite{Brzezinski-Curveddifferentialgraded-2013} generalising an earlier result of Ro\v{\i}ter in \cite{Roiter-Matrixproblems-1979}. We start by recalling the definition of a curved differential graded algebra.

\begin{definition}
\begin{enumerate}[(i)]
\item A \emphbf{curved differential graded algebra} is an $\mathbb{N}$-graded algebra $\mathcal{A}=\bigoplus_{n\in \mathbb{N}} \mathcal{A}^{(n)}$ together with a $k$-linear map  $d\colon \mathcal{A}\to \mathcal{A}$ of degree $1$ and an element $\gamma\in \mathcal{A}^{(2)}$ satisfying the following condition:
\begin{enumerate}[{(CDGA}1{)}]
\item $d$ satisfies the \emphbf{graded Leibniz rule}, i.e. 
\[d(ab)=d(a)b+(-1)^{|a|}ad(b)\] for all homogeneous $a,b\in \mathcal{A}$.
\item For all $a\in \mathcal{A}$, $d^2(a)=\gamma a-a\gamma$.
\item $d(\gamma)=0$.
\end{enumerate}
\item A curved differential graded algebra is called \emphbf{semi-free} if $\mathcal{A}\cong T_B(W)$ as graded algebras, where $B$ is an algebra and $W$ is an $A$-bimodule and the grading on the tensor algebra $T_B(W)$ is given by the usual tensor grading where $B$ has degree $0$ and $W$ is of degree $1$.
\item A curved differential graded algebra is called a \emphbf{differential graded algebra} if $\gamma=0$.
\end{enumerate}
\end{definition}

For stating Ro\v{\i}ter's original result, we have to recall the definition of a normal bocs.

\begin{definition}
Let $B$ be a basic algebra with idempotents $e_1,\dots, e_n$. A bocs $(B,W)$ is called \emphbf{normal} if there is a set $\mathfrak{w}=\{w_i\in W(i,i)| \varepsilon(w_i)=e_i\}$ such that $\mu(w_i)=w_i\otimes w_i$.
\end{definition}

The second part of the following theorem is due to Ro\v{\i}ter while its generalisation \cite[Corollary 3.12]{Brzezinski-Curveddifferentialgraded-2013} is due to Brzezi{\'n}ski \cite{Brzezinski-Curveddifferentialgraded-2013}.

\begin{theorem}
There is an equivalence of categories between the category of bocses with a surjective counit and the category of semi-free curved differential graded algebras. It restricts to an equivalence between the category of normal bocses and the category of semi-free differential graded algebras.
\end{theorem}

The next subsection gives a rough outline on how to construct the bocs $\mathfrak{B}=(B,W)$ with a right algebra $R$ starting from a quasi-hereditary algebra $A$ in such a way that $R$ is Morita equivalent to $A$.

\subsection{Bocses via $A_\infty$-categories}

In the previous subsection we have seen the relation of bocses to differential graded algebras, which are differential graded categories with one object. In this subsection we are concerned with the relation of bocses to $A_\infty$-categories. $A_\infty$-categories are a generalisation of differential graded categories. One major advantage upon differential graded categories is that the homology of an $A_\infty$-category is a (quasi-isomorphic) $A_\infty$-category. For a general introduction into $A_\infty$-categories, the reader can consult the three survey articles \cite{Keller-Introduction-2001, Keller-Ainfinityalgebras-2002, Keller-Ainfinityalgebras-2006} of Keller dealing with different aspects of the theory. The results stated in these surveys follow from the abstract theory developed in \cite{LefevreHasegawa-SurlesAinfini-2003}. Other introductions to $A_\infty$-categories can be found in \cite[Appendix B]{Madsen-Homologicalaspects-2002}, \cite{LuPalmieriWuZhang-Ainfinityalgebras-2004} and \cite[Chapter I]{Seidel-Fukayacategories-2008}. The reader should be warned that the different authors deal with different sign conventions.

\begin{definition}
A (small) \emphbf{$A_\infty$-category} $\mathcal{A}$ consists of
\begin{description}
\item[objects:] a set of objects $\mathcal{A}$,
\item[morphisms:] ror each pair of objects $x,y\in \mathcal{A}$ a $\mathbb{Z}$-graded $k$-vector space of morphisms $\mathcal{A}(x,y)$,
\item[multiplications:] for each $n$ and each sequence of objects $x_1,\dots,x_n\in \mathcal{A}$ a $k$-linear map $m_n\colon \mathcal{A}(x_{n-1},x_n)\otimes_k \dots\otimes_k \mathcal{A}(x_1,x_2)\to \mathcal{A}(x_1,x_n)$ of degree $2-n$ such that the following conditions hold for all $n\in \mathbb{N}$:
\[(A_n)\quad \sum_{\substack{n=r+s+t\\u=r+1+t}}(-1)^{r+st}m_u(1^{\otimes r}\otimes m_s\otimes 1^{\otimes t})=0.\]
\end{description}
\end{definition}

\begin{remark}
\begin{enumerate}[(i)]
\item The condition $(A_1)$ states that $m_1\circ m_1=0$, so that $\mathcal{A}(x,y)$ can be regarded as a complex and one can define its homology $H^*\mathcal{A}$.
\item The condition $(A_2)$ is just the graded Leibniz rule for $m_2$ and the differential $m_1$.
\item An $A_\infty$-category with $m_n=0$ for $n\geq 3$ is precisely the same as a differential graded category.
\item Note that $m_2$ is in general not associative for an $A_\infty$-category, but it is associative up to a homotopy given by $m_3$ as one can see from the condition for $n=3$. In the examples we consider $\mathcal{A}$ will either be differential graded so that $m_3=0$ or it will be \emphbf{minimal}, i.e. $m_1=0$. In these cases, condition $(A_3)$ just gives associativity of the algebra.
\end{enumerate}
\end{remark}

Morphisms in the category of (small) $A_\infty$-categories are defined as follows:

\begin{definition}
\begin{enumerate}[(i)]
\item A morphism $f\colon \mathcal{A}\to \mathcal{B}$ between two (small) $A_\infty$-ca\-te\-go\-ries $\mathcal{A}$ and $\mathcal{B}$ is given by an object $f(x)\in \mathcal{B}$ for each object $x\in \mathcal{A}$ and for all $n\geq 1$ maps $f_n\colon \mathcal{A}(x_{n-1},x_n)\otimes_k \dots\otimes_k \mathcal{A}(x_1,x_2)\to \mathcal{B}(f(x_1),f(x_n))$ of degree $1-n$ such that for all $n\geq 1$ the following equations hold:
\[\sum_{\substack{n=r+s+t\\u=r+1+t}} (-1)^{r+st}f_u(1^{\otimes r}\otimes_k m_s\otimes_k 1^{\otimes t})=\sum_{\substack{1\leq r\leq n\\n=i_1+\dots+i_r}}(-1)^sm_r(f_{i_1}\otimes_k \dots\otimes f_{i_r})\]
where $s=(r-1)(i_1-1)+(r-2)(i_2-1)+\dots+2(i_{r-2}-1)+(i_{r-1}-1)$. 
\item The morphism $f$ is called a \emphbf{quasi-isomorphism} if $f_1$ is a quasi-isomorphism of complexes.
\end{enumerate}
\end{definition}

The following theorem, usually attributed to Kadei{\v{s}}vili (see \cite{Kadeishvili-Thealgebraicstructure-1983}, but also \cite{Kadeishvili-Onthetheory-1980, Smirnov-Homology-1980, Proute-Algebresdifferentielles-1984, GugenheimLambeStasheff-Algebraicaspects-1990, JohanssonLambe-Transferringalgebrastructures-2001, Merkulov-Stronghomotopyalgebras-1999}), states the existence of a minimal model for every small $A_\infty$-category, i.e. a quasi-isomorphic $A_\infty$-category with $m_1=0$.

\begin{theorem}
Let $\mathcal{A}$ be a small $A_\infty$-category. Then, there is an $A_\infty$-structure on $H^*\mathcal{A}$ with $m_1=0$ and $m_2$ is induced by the $m_2$ on $\mathcal{A}$, such that there is a quasi-isomorphism of $A_\infty$-categories $H^*\mathcal{A}\to \mathcal{A}$ lifting the identity of $H^*\mathcal{A}$.
\end{theorem}

There are two constructions for computing this $A_\infty$-structure on $H^*\mathcal{A}$. Since we only need it in this case, we specialise to $\mathcal{A}$ being a differential graded category.

\begin{remark}
\begin{enumerate}[(i)]\label{Merkulov}
\item The first construction is due to Merkulov in \cite{Merkulov-Stronghomotopyalgebras-1999}. There are several possible choices one can make in this construction. We follow the choices of \cite{LuPalmieriWuZhang-Ainfinitystructure-2009}.   Let $\mathcal{A}$ be a differential graded category with differential $d$. Let $\mathcal{Z}$ be the cycles of $\mathcal{A}$ and $\mathcal{B}$ be the boundaries. Identify the homology of $\mathcal{A}$ with a subspace $\mathcal{H}$ of $\mathcal{A}$. As we work over a field, we can find a subspace $\mathcal{L}$ of $\mathcal{A}$ with $\mathcal{A}=\mathcal{B}\oplus \mathcal{H}\oplus \mathcal{L}$. Let $p\colon \mathcal{A}\to \mathcal{H}$ be the projection on $\mathcal{H}$ and let $G\colon \mathcal{A}\to \mathcal{A}$ be the degree $-1$ map with $Q|_{\mathcal{L}\oplus \mathcal{H}}=0$ and $Q|_{\mathcal{B}}=(d|_{\mathcal{L}})^{-1}$. Let $\lambda_n\colon \mathcal{A}^{\otimes n}\to \mathcal{A}$ for $n\geq 2$ be defined recursively as $\lambda_2(a_1,a_2)=m_2(a_1, a_2)$ and 
\[\lambda_n(a_1,\dots,a_n)=-\sum_{\substack{k+l=n\\k,l\geq 1}}(-1)^\sigma m_2(G(\lambda_k(a_1,\dots,a_k)),G(\lambda_l(a_{k+1},\dots,a_n))),\]
where $\sigma=k+(l-1)\left(\sum_{i=1}^{k}|a_i|\right)$, and $G\lambda_1:=-1$ by convention. Then a minimal model is defined on $\mathcal{H}$ by $m_1=0$ and $m_n=p\lambda_n i$, where $i\colon \mathcal{H}\to \mathcal{A}$ is the canonical inclusion. For an example of this construction, see page \pageref{exampleMerkulov}.
\item The second construction can be found in Keller's paper \cite{Keller-Ainfinityalgebras-2002}. Again let $\mathcal{A}$ be a differential graded category with differential $d$. Define $m_i\colon H^*\mathcal{A}\to H^*\mathcal{A}$ of degree $2-i$ inductively as follows: $m_1=0$ and $m_2$ is the multiplication induced by the multiplication of $\mathcal{A}$. By definition, $m_2(f_1\otimes f_1)$ and $f_1m_2$ are homotopic as morphisms of complexes. Choose $f_2\colon H^*\mathcal{A}\otimes H^*\mathcal{A}\to \mathcal{A}$ as a morphism of complexes of degree $-1$ such that $f_1m_2=m_1f_2+m_2(f_1\otimes f_1)$. Let $\Phi_3=m_2(f_1\otimes f_2-f_2\otimes f_1)+f_2(1\otimes m_2-m_2\otimes 1)$. Then, by a similar argument, there exist $f_3\colon (H^*\mathcal{A})^{\otimes 3}\to \mathcal{A}$ and $m_3\colon (H^*\mathcal{A})^{\otimes 3}\to H^*\mathcal{A}$ such that $f_1m_3=m_1f_3+\Phi_3$. Continueing in a similar fashion one can construct $f_i$ and $m_i$ for all $i\geq 3$.
\end{enumerate}
\end{remark}

The case in which we apply this construction is the following: Let $A$ be an algebra and let $M_1,\dots, M_n$ be $A$-modules with projective resolutions $P_1^*,\dots, P_n^*$. Let $P=\bigoplus_{i,j} P_i^{j}$. Then, $\Hom_k(P,P)$ has the structure of a differential graded category $\mathcal{D}$ with objects $1,\dots,n$ and morphism spaces $\mathcal{D}(i,i')$ of degree $k$ given by $\Hom_k(\bigoplus_{j}P_i^j,\bigoplus_{j}P_{i'}^{j+k})$ and with differential $d$ given by $d(f)=f\circ \partial-(-1)^{|f|}\partial \circ f$, where $\partial$ is the differential on $P$. Composition is given by composition of maps. The homology of this complex can then be identified with $\Ext^*_A(M,M)$ where $M=\bigoplus_{i=1}^n M_i$. We are now ready to state Keller's reconstruction theorem for the quiver and relations of an algebra. The dual version of this statement is stated without proof in \cite[Proposition 2]{Keller-Ainfinityalgebras-2002}, in a different setting this result is proved in \cite[Theorem A]{LuPalmieriWuZhang-Ainfinitystructure-2009}.  For simplicity of exposition we restrict to the case of an acyclic quiver where the result also follows from \cite{KoenigKulshammerOvsienko-Quasihereditaryalgebras-2014} in the special case where the standard modules are simple. 

\begin{theorem}\label{Kellerreconstruction}
Let $A$ be a basic algebra with acyclic Gabriel quiver $Q$. Let $L_1,\dots,L_n$ be the simple modules (up to isomorphism). Let $L=\bigoplus_{i=1}^n L_i$ and let the $\Ext^*_A(L,L)$ be equipped with the $A_\infty$-structure constructed before. Let $m=(m_i)_{i\geq 2}\colon \bigoplus_{i\geq 2} (\Ext^1_A(L,L))^{\otimes i}\to \Ext^2_A(L,L)$ and let $d=Dm$ be the $k$-dual map. Then, an ideal $I$ such that $A\cong KQ/I$ is generated by $\image(Dm)$. In particular, the number of relations from $i$ to $j$ is given by $\dim D\Ext^2(L_i,L_j)$.
\end{theorem}

In the case of a quasi-hereditary algebra $A$, the idea of how to construct a corresponding bocs $\mathfrak{B}=(B,W)$ is as follows: the algebra $B$ can be obtained by disregarding the homomorphisms between the $\Delta_1,\dots,\Delta_n$. Let $\Delta:=\bigoplus_{i=1}^n \Delta_i$. Applying Keller's reconstruction theorem to $\bigoplus_{i=1}^n k1_i\oplus \bigoplus_{j\geq 1} \Ext^j_A(\Delta,\Delta)$ yields an algebra $B$. The coring $W$ over $B$ is then constructed in a similar way using $\Hom_A(\Delta,\Delta)$ as well. In particular, $\overline{W}=B\otimes_k \rad(\Delta,\Delta)\otimes_k B$, where $\rad(\Delta,\Delta)$ is given by the non-isomorphisms between the $\Delta_i$. For the precise construction see the paper \cite{KoenigKulshammerOvsienko-Quasihereditaryalgebras-2014}.

\subsection{Bocses via differential biquivers}\label{bocsesviadifferentialbiquivers}

For finite dimensional algebras, hereditary algebras are given by path algebras of acyclic quivers. In the context of bocses, the bocses with hereditary module category are given by the differential biquivers. The quasi-hereditary algebras corresponding to bocses arising from a directed differential biquiver are precisely the strongly quasi-hereditary algebras, a notion coined by Ringel in \cite{Ringel-Iyamasfinitenesstheorem-2010}, but studied much earlier, see e.g. \cite{DlabRingel-Themoduletheoreticalapproach-1992}.

\begin{definition}
\begin{enumerate}[(i)]
\item A \emphbf{biquiver} is a quiver $(Q_0,Q_1)$ with two kinds of arrows, i.e. the arrows arise as a disjoint union $Q_1=Q_1^0\cup Q_0^0$. The arrows in $Q_1^0$ are called \emphbf{solid} while the arrows in $Q_1^1$ are called \emphbf{dashed}.
\item The path algebra $kQ$ is regarded as a graded algebra with $\deg Q_1^0=0$ and $\deg Q_1^1=1$.
\item A pair $(Q,\partial)$ consisting of a biquiver and a linear map $\partial\colon kQ\to kQ$ is called a \emphbf{differential biquiver} if $\partial$ satisfies the following properties:
\begin{enumerate}[{(D}1{)}]
\item $\partial$ is of degree $1$ and $\partial^2=0$.
\item $\partial(e)=0$ for $e$ a trivial path
\item $\partial$ satisfies the graded Leibniz rule.
\end{enumerate}
\end{enumerate}
\end{definition}

\begin{definition}
Let $(Q,\partial)$ be a differential biquiver. Then the category of \emphbf{representations} of $(Q,\partial)$ is given as follows:
\begin{description}
\item[objects:] representations of $kQ^0$ where $Q^0=(Q_0,Q_1^0)$.
\item[morphisms:] Given two representations of $kQ^0$, $V$ and $W$, a morphism from $V$ to $W$ is a family of $k$-linear maps $f_i\colon V_i\to W_i$ for each vertex $i\in Q_0$ together with a family of $k$-linear maps $f_\varphi\colon V_i\to W_j$ for each arrow $\varphi\colon\begin{tikzcd}i\arrow[dashed]{r}&j\end{tikzcd}$ such that the following is satisfied: If $\partial(a)=\sum_{\varphi}\lambda_\varphi b_\varphi \varphi b_\varphi'$ for $\lambda_{\varphi}\in k$, $b_\varphi, b_\varphi'\in kQ^0$, $\varphi\in Q_1^1$, then $\sum W_{b_\varphi} f_\varphi W_{b_\varphi'}=W_{a} f_i-f_j V_a$ for all solid arrows $a\colon i\to j$.
\item[composition:] If $\partial(\varphi)=\sum_{\psi,\psi'} \lambda_{\psi,\psi'} c_{\psi,\psi'} \psi c'_{\psi,\psi'}\psi'c''_{\psi,\psi'}$ with $\varphi\colon \begin{tikzcd}i\arrow[dashed]{r}&j\end{tikzcd}\in Q_1^1$, $\psi,\psi'\in Q_1^1$, $\lambda_{\psi,\psi'}\in k$, $c_{\psi,\psi'},c_{\psi,\psi'}',c_{\psi,\psi'}''\in kQ^0, $, then $(f\circ g)_i=f_i\circ g_i$ and $(f\circ g)_\varphi=f_j g_\varphi + f_\varphi g_i + \sum_{\psi,\psi'} \lambda_{\psi,\psi'} W_{c_{\psi,\psi'}} f_\psi V_{c_{\psi,\psi'}'} g_{\psi'} U_{c_{\psi,\psi'}'}$.
\item[units:] The identity morphism of a representation $V$ is given by $f_i=1_{V_i}\colon V_i\to V_i$ and $f_\varphi=0$ for all $\varphi\in Q_1^1$.
\end{description}
\end{definition}

\begin{example}
Let $Q$ be the biquiver $\begin{tikzcd}1\arrow[yshift=0.5ex]{r}{a}\arrow[yshift=-0.5ex, dashed]{r}[swap]{\varphi}\arrow[yshift=0.5ex, bend left]{rr}{b}\arrow[yshift=-0.5ex, bend right, dashed]{rr}[swap]{\chi}&2\arrow[dashed]{r}{\psi}&3\end{tikzcd}$ with non-zero differential on the arrows given by $\partial(\chi)=\psi\varphi$ and $\partial(b)=\psi a$. A representation $V$ is given by vector spaces $V_1,V_2,V_3$ and linear maps $V_a\colon V_1\to V_2$ and $V_b\colon V_1\to V_3$. Given another representation $W$, a morphism $g\colon V\to W$ is given by $6$ linear maps $g_1\colon V_1\to W_1$, $g_2\colon V_2\to W_2$, and $g_3\colon V_3\to W_3$  (compare with the notion of a representation of a quiver) and $g_\varphi\colon V_1\to W_2$, $g_\psi\colon V_2\to W_3$, and $g_\chi\colon V_1\to W_3$ satisfying $g_2V_a=W_ag_1$ (because $\partial(a)=0$) and $g_3V_b-W_bg_1+g_\psi V_a=0$ (because $\partial(b)=\psi a$). The composition $h:=g\circ f$ with a map $f\colon U\to V$ is given by the maps $h_i=g_if_i$ for $i\in \{1,2,3\}$ and $h_\psi=g_3f_\psi+g_\psi f_2$ (because $\partial(\psi)=0$), $h_\varphi=g_2f_\varphi+g_\varphi f_1$ (because $\partial(\varphi)=0$) and $h_\chi:=g_3f_\chi+g_\chi f_1+g_\psi f_\varphi$ (since $\partial(\chi)=\psi\varphi$).
\end{example}

\begin{definition}
A bocs $(B,W)$ with projective kernel is called \emphbf{free} if $B$ is a hereditary algebra.
\end{definition}

The following theorem is already due to Roiter and contained in \cite{Roiter-Matrixproblems-1979}, similar ideas are also contained in \cite{KleinerRoiter-Representations-1975} before bocses had been introduced. For a precise reference, see e.g. \cite[Theorem 6.3.5]{Bodnarchuk-Vectorbundles-2007}.

\begin{theorem}
There is a one-to-one correspondence between free normal bocses and differential biquivers given as follows:
\begin{enumerate}[(i)]
\item If $(Q,\partial)$ is a differential biquiver, a free normal bocs can be defined as $B=kQ_1^0$ and $\overline{W}=B\otimes_k kQ_1^1\otimes_k B$. There is a $B$-$B$-bimodule structure on $W=B\oplus \overline{W}$ by the following two maps $\begin{pmatrix}1\\0\end{pmatrix}\colon B\to B\oplus \overline{W}$, the canonical inclusion, for the left structure and $\begin{pmatrix}1\\\partial\end{pmatrix}\colon B\to B\oplus \overline{W}$ for the right structure.
\item If $(B,W)$ is a free normal bocs. Then a biquiver can be obtained by the graded algebra homomorphism $kQ\cong T_B(\overline{W})$, where the tensor algebra $T_B(\overline{W})$ is given the usual grading where $B$ has degree $0$ and $W$ has degree $1$. The normal section then defines a map $T_B(\overline{W})\to T_B(\overline{W})$ by linear extension of the map given on $A$ by $\partial(a)=aw_i-w_ja$ and $\partial(v)=\mu(v)-v\otimes w_i-w_j\otimes v$.
\end{enumerate}
Furthermore, under this bijection the category of modules over the free normal bocs is equivalent to the category of representations of the differential biquiver. 
\end{theorem}

Connecting this to quasi-hereditary algebras, these classes correspond to the following class of quasi-hereditary algebras:

\begin{definition}
A quasi-hereditary algebra is called \emphbf{strongly quasi-hereditary} if $\projdim \Delta(i)\leq 1$ for all $i$.
\end{definition}

\begin{proposition}
An algebra $A$ is strongly quasi-hereditary if and only if it is Morita equivalent to the right algebra of a free normal bocs.
\end{proposition}

\subsection{The reduction algorithm}

The reduction algorithm is a way to construct from a differential biquiver a series of differential biquivers with equivalent categories of representations. The goal is to decrease the complexity of the solid part (as a drawback one has to increase the complexity of the dashed part). In the representation-finite case one ultimately reaches a differential biquiver with no solid arrows (and as many dashed arrows as the dimension of the Jacobson radical of its Auslander algebra). It thus gives a way to enhance the one-to-one correspondence of modules over an algebra with semi-simple modules over its Auslander algebra to an equivalence of categories. There are many different formulations and levels of generality for this algorithm. A first version of a reduction algorithm for bocses already appears in \cite{KleinerRoiter-Representations-1975}. Its full strength was used in Drozd's proof of the tame--wild dichotomy theorem \cite{Drozd-Tameandwild-1977, Drozd-Tameandwild-1979, Drozd-Tameandwild-1980}, see also \cite{CrawleyBoevey-Ontamealgebras-1988, BautistaSalmeronZuazua-Differentialtensoralgebras-2009} for later formulations. In this subsection, we mostly follow \cite{Bodnarchuk-Vectorbundles-2007} which contains a precise formulation of part of the reduction algorithm in the language of differential biquivers. For our application to Schur algebras, we also need a result of Bautista and Salmer\'on on differential biquivers with relations, see \cite{BautistaSalmeron-TheKievalgorithm-1989}.\\
The first part of the reduction algorithm comes from the following example of an equivalence of categories of representations of differential biquivers:

\begin{example}
Let $\begin{tikzcd}1\arrow[yshift=0.5ex]{r}{a}\arrow[dashed, yshift=-0.5ex]{r}[swap]{\varphi}&2\end{tikzcd}$ be a differential biquiver with $\partial(a)=\varphi$. Then, it is easy to see that the extension given by the solid arrow $1\to 2$ actually splits in the category of representations of the differential biquiver, i.e. there is the following isomorphism of representations:
\[\begin{tikzcd}
k\arrow{d}[swap]{1}\arrow[dashed]{rd}{1}\arrow[dashed]{r}{1}& k\arrow{d}{0}\\
k\arrow[dashed]{r}[swap]{1}&k
\end{tikzcd}\]
Hence, there is the following equivalence of categories:
\[\operatorname{rep}(\begin{tikzcd}1\arrow[yshift=0.5ex]{r}{a}\arrow[dashed, yshift=-0.5ex]{r}[swap]{\varphi}&2\end{tikzcd},\partial)\cong \operatorname{rep}(\begin{tikzcd}1&2\end{tikzcd}).\]
\end{example}

This motivates the following definition:

\begin{definition}
Let $(Q,\partial)$ be a differential biquiver. A solid arrow $a\colon i\to j$ is called \emphbf{non-regular} or \emphbf{superfluous} if $\partial(a)=\lambda v+\sum_{i}\mu_i p_i$ where the $p_i$ are paths $i\to j$ with $v$ not contained in any of them.
\end{definition}

The equivalence given in the example can now be applied locally for an arbitrary differential biquiver containing a superfluous arrow.

\begin{proposition}\label{regularisation}
Let $a$ be a non-regular arrow of a differential biquiver with $\partial(a)=\lambda v+\sum_{i}\lambda_i p_i$. Define a new biquiver $\tilde{Q}$ with the same vertex set and $\tilde{Q}^0:=Q_0\setminus \{a\}$ and $\tilde{Q}^1_1=Q^1_1\setminus \{v\}$. Let $\tilde{\partial}\colon k\tilde{Q}\to k\tilde{Q}$ be the $k$-linear map obtained from $\partial$ be substituting $v$ by $-\lambda^{-1}(\sum_{i}\lambda_ip_i)$. Then $(\tilde{Q},\tilde{\partial})$ is a differential biquiver and the categories of representations of $(Q,\partial)$ and $(\tilde{Q},\tilde{\partial})$ are equivalent.
\end{proposition}

The second part of the reduction algorithm is based on the Gau{\ss} algorithm, or equivalently, the representation theory of the hereditary algebra $kA_2$.

\begin{example}
Consider the bocs $\mathfrak{B}$ given by the following differential biquiver $\begin{tikzcd}1&3\arrow[dashed]{l}{\pi}&2\arrow[dashed]{l}{\iota}\end{tikzcd}$ with zero differential. Then, there is an equivalence of categories 
\[F\colon \modu kA_2 \to \modu \mathfrak{B}\]
given as follows: Each $kA_2$-module can be decomposed into indecomposables $M\cong S_1^{m_1}\oplus S_2^{m_2}\oplus P_1^{m_3}$. On objects, $F$ is then given by sending $M$ to the representation with $k^{m_1}$ on the vertex $1$, $k^{m_2}$ on the vertex $2$ and $k^{m_3}$ on the vertex $3$. Furthermore each map $M\to N$ can be written as a matrix 
\[\begin{tikzcd}[column sep=10ex, ampersand replacement=\&]{S_1^{m_1}\oplus S_2^{m_2}\oplus P_1^{m_3}}\arrow{r}{\begin{pmatrix}f_1 &0&g_2\\0&f_2&0\\0&g_1&f_3\end{pmatrix}} \&{S_1^{n_1}\oplus S_2^{n_2}\oplus P_1^{n_3}}\end{tikzcd}\]
Noting that $\Hom(S_1^{m_1},S_1^{n_1})\cong (\Hom(S_1,S_1))_{m_1\times n_1}\cong k^{m_1\times n_1}$ one can identify $f_1$ with an $m_1\times n_1$-matrix. Similarly for $f_2$, $f_3$, $g_1$, and $g_2$. Hence each $kA_2$-linear map $M\to N$ can be identified with five $k$-linear maps, i.e. with a morphism $FM\to FN$ in $\modu \mathfrak{B}$ as follows:
\[
\begin{tikzcd}
k^{m_2}\arrow[dashed]{r}{f_2}\arrow[dashed]{rd}{g_1}&k^{n_2}\\
k^{m_3}\arrow[dashed]{r}{f_3}\arrow[dashed]{rd}{g_2}&k^{n_3}\\
k^{m_1}\arrow[dashed]{r}{f_1}&k^{n_1}
\end{tikzcd}
\]
\end{example}

The second part of the reduction algorithm, the minimal edge reduction, does the replacement of the previous example locally for a differential biquiver.

\begin{proposition}\label{minimaledgereduction}
Let $(Q,\partial)$ be a differential biquiver. Let $a\colon i\to j$ be a solid arrow with $\partial(a)=0$. Define a differential biquiver $(\tilde{Q},\tilde{\partial})$ as follows:
\begin{description}
\item[vertices:] $\tilde{Q}$ has one vertex more than $Q$, call it $\emptyset$.
\item[arrows:] 
\begin{enumerate}[(a)]
\item For every arrow $b\colon l\to m$ of $Q$ different from $a$ add an arrow $b\colon l\to m$ into $\tilde{Q}$. If the original arrow was solid (resp. dashed), the new one is solid (resp. dashed).
\item For each arrow $b\colon l\to m$ where $l\in \{i,j\}$, add an arrow $b_{m,\emptyset}\colon \emptyset\to m$ to $\tilde{Q}$. Again if the original arrow was solid (resp. dashed), the new one is solid (resp. dashed).
\item For each arrow $b\colon l\to m$ where $m\in \{i,j\}$, add an arrow $b_{\emptyset,l}\colon l\to \emptyset$ to $\tilde{Q}$. The same rule for solid (resp. dashed) applies.
\item For each arrow $b\colon l\to m$ where $l,m\in \{i,j\}$, add a loop $b_{\emptyset,\emptyset}\colon \emptyset\to \emptyset$ to $\tilde{Q}$. The same rule for solid (resp. dashed) applies.
\item Add two new dashed arrows, $\iota\colon \begin{tikzcd}j\arrow[dashed]{r} &\emptyset\end{tikzcd}$ and $\pi\colon \begin{tikzcd}\emptyset\arrow[dashed]{r} &i\end{tikzcd}$.
\end{enumerate}
\item[differential:] Set $\tilde{\partial}(\iota)=\tilde{\partial}(\pi)=0$.\\
Define $F\colon Q_1\to \tilde{Q}_1\cup M_{2\times 1}(\tilde{Q}_1)\cup M_{1\times 2}(\tilde{Q}_1)\cup M_{2\times 2}(\tilde{Q}_1)\cup M_{2\times 2}(k)$ as follows:\\
$F(b)=b$ for an arrow in case (a).\\
If $b\colon i\to m$ for $m\notin \{i,j\}$, then $F(b)=(b_{m,\emptyset},b)$.\\
If $b\colon l\to i$ for $l\notin \{i,j\}$, then $F(b)=\begin{pmatrix}b_{\emptyset,m}\\b\end{pmatrix}$.\\
If $b\colon j\to m$ for $m\notin \{i,j\}$, then $F(b)=(b,b_{m,\emptyset})$.\\
If $b\colon l\to j$ for $l\notin \{i,j\}$, then $F(b)=\begin{pmatrix}b\\b_{\emptyset,l}\end{pmatrix}$.\\
If $b\colon i\to i$, then $F(b)=\begin{pmatrix}b_{\emptyset,\emptyset}&b_{\emptyset,1}\\b_{1,\emptyset}&b\end{pmatrix}$.\\
If $b\colon j\to j$, then $F(b)=\begin{pmatrix}b&b_{j,\emptyset}\\b_{\emptyset,j}&b_{\emptyset,\emptyset}\end{pmatrix}$.\\
If $b\colon i\to j$, different from $a$, then $F(b)=\begin{pmatrix}b_{j,\emptyset}&b\\b_{\emptyset,\emptyset}&b_{\emptyset,i}\end{pmatrix}$.\\
Define $F(a)=\begin{pmatrix}0&0\\1&0\end{pmatrix}$.\\
If $b\colon j\to i$, then $F(b)=\begin{pmatrix}b_{\emptyset,j} &b_{\emptyset,\emptyset}\\b&b_{i,\emptyset}\end{pmatrix}$.\\
Define the new differential by $\tilde{\partial}(F(x)):=F(\partial(x))$ for all $x\in Q_1\setminus \{a\}$.
\end{description}
Then, the categories of representations of $(Q,\partial)$ and $(\tilde{Q},\tilde{\partial})$ are equivalent.
\end{proposition}

In the representation-finite case, the regularisation and minimal edge reduction suffice:

\begin{proposition}
Let $\mathfrak{B}=(B,W)$ be a free normal directed bocs with represen\-ta\-tion-finite category of modules. Then, applying Proposition \ref{regularisation} as long as possible followed by Proposition \ref{minimaledgereduction} one arrives at a free normal bocs with no solid arrows, as many vertices as indecomposable modules in $\modu \mathfrak{B}$ and as many dashed arrows as the dimension of the Jacobson radical of the Auslander algebra.
\end{proposition}

In joint work with Ulrich Thiel \cite{KulshammerThiel-Boxes-2016}, we implemented this algorithm using the computer algebra system MAGMA. This is also what is behind our calculations in the forthcoming examples subsection.\\
For tame or wild algebras, one needs one additional part of the reduction algorithm, the so called partial loop reduction (for a motivation on this part, see e.g. \cite{CrawleyBoevey-Matrixproblems-1990}). Since it is not needed in our study of representation-finite categories of filtered modules, we will not repeat it here and only mention a straightforward corollary of it, namely:

\begin{corollary}\label{stronglyfilteredtamewild}
Let $A$ be a strongly quasi-hereditary algebra. Then, the category of modules filtered by standard modules is either representation-finite, tame, or wild.
\end{corollary}

\begin{proof}
The category $\mathcal{F}(\Delta)$ for a strongly quasi-hereditary algebra is equivalent to the category $\modu \mathfrak{B}$ for a free normal directed bocs $\mathfrak{B}=(B,W)$. By the proof of Drozd's tame--wild dichotomy theorem for algebras, the module category of such a bocs is either representation-finite, tame, or wild.
\end{proof}

Next, we recall differential biquivers with relations following Bautista and Salmer\'on. 

\begin{lemma}
Let $\mathfrak{B}=(B,W)$ be a bocs and let $I\subset B$ be an ideal. Define $\tilde{B}=B/I$ and $\tilde{W}=\tilde{B}\otimes_B W\otimes_B \tilde{B}$ and let $\tilde{\mu}$ and $\tilde{\varepsilon}$ be induced by $\mu$ and $\varepsilon$, respectively. Then $\tilde{\mathfrak{B}}=(\tilde{B},\tilde{W})$ is a bocs whose module category can be identified with the full subcategory of $\modu \mathfrak{B}$ of modules that vanish on $I$. 
\end{lemma}

For this version of the reduction algorithm some compatibility with the set of generators of $(B,W)$ is needed:

\begin{definition}
Let $\mathfrak{B}=(B,W)$ be a normal bocs with an ideal $I\subset B$. Then $I$ is called \emphbf{compatible} with the set of grouplikes $\mathfrak{w}$ if there is a finite generating set $b_1,\dots,b_t$ of $I$ as a $B$-$B$-bimodule such that $\partial(b_i)\in I_{i-1}W+WI_{i-1}$ for all $0<i\leq t$ where $I_i$ is the $B$-$B$-subbimodule of $I$ generated by $b_1,b_2,\dots,b_i$ for $0\leq i\leq t$.  
\end{definition}

\begin{theorem}[{\cite[Theorem 7]{BautistaSalmeron-TheKievalgorithm-1989}}]\label{BautistaSalmeronreduction}
Let $\mathfrak{B}=(B,W)$ be a normal directed bocs with an ideal $I\subset B$ compatible with the set of grouplikes. Let $\tilde{B}$ be given as in \ref{minimaledgereduction}. Let $I_F$ be the ideal of $\tilde{B}$ generated by $F(I)$. Then, $I_F$ is compatible with the set of grouplikes for $\tilde{\mathfrak{B}}$.\\
In the representation-finite case, this algorithm stops at a free normal bocs with no solid arrows, as many vertices as indecomposable modules in $\modu \mathfrak{B}$ and as many dashed arrows as the dimension of the Jacobson radical of the Auslander algebra.
\end{theorem}

\subsection{Passing to the hereditary situation}

This subsection is just an excursion to indicate one step which is missing to prove Drozd's tame--wild dichotomy theorem. Although in the last section, we have seen a version of the reduction algorithm for bocses "with relations" only in the representation-finite case. There is also no version of the full version of the reduction algorithm for bocses where $B$ is non-hereditary. Instead in the proof of Drozd's tame--wild dichotomy theorem, one uses the following trick to pass to a hereditary bocs with almost equivalent category of representations. This is described in detail in \cite[Section 18]{BautistaSalmeronZuazua-Differentialtensoralgebras-2009}.

\begin{definition}
Let $A$ be a finite dimensional algebra. 
\begin{enumerate}[(i)]
\item Define $\mathcal{P}(A)$ to be the category with
\begin{description}
\item[objects:] triples $(P,Q,f)$, where $P,Q\in A-\proj$ and $f\colon P\to Q$ is an $A$-module homomorphism
\item[morphisms:] Homomorphisms from $(P,Q,f)$ to $(P,Q,f')$ are given by pairs $(\alpha,\beta)$, where $\alpha\colon P\to P'$ and $\beta\colon Q\to Q'$ such that the following diagram commutes:
\[\begin{tikzcd}
P\arrow{r}{f}\arrow{d}{\alpha}&Q\arrow{d}{\beta}\\
P'\arrow{r}{f'}&Q'.
\end{tikzcd}\]
\end{description}
\item Let $\mathcal{P}^1(A)$ be the full subcategory with $\image f\subseteq \rad Q$.
\item Let $\mathcal{P}^2(A)$ be the full subcategory with $\Ker f\subseteq \rad P$ and $\image f\subseteq \rad Q$.
\end{enumerate}
\end{definition}

\begin{theorem}
\begin{enumerate}[(i)]
\item Every object $X\in \mathcal{P}(A)$ is of the form $X\cong (P,0,0)\oplus (Q,Q,1_Q)\oplus \overline{X}$ for some $\overline{X}\in \mathcal{P}^2(A)$ and some $P,Q\in A-\proj$. 
\item The restriction of the cokernel functor $\Cok\colon \mathcal{P}(A)\to \modu A, (P,Q,f)\mapsto \Cok f$ to $\mathcal{P}^2(A)$ is dense, full, reflects isomorphisms, and preserves and reflects indecomposability.
\end{enumerate}
\end{theorem}

The categories $\mathcal{P}^2(A)$ and $\mathcal{P}^1(A)$ only differ in the objects given by $(P,0,0)$. Therefore, determining the representation type of $\mathcal{P}^2(A)$ is the same as determining the representation type of $\mathcal{P}^1(A)$. This category is equivalent to representations of a differential biquiver as follows:

\begin{proposition}
Let $A$ be a finite dimensional algebra with $n$ simples. Fix a basis of $\rad(P_i,P_j)$ for each pair of projective modules $(P_i,P_j)$. Define a differential biquiver on $2n$ vertices, denoted $\{1,\dots, n,\overline{1},\dots,\overline{n}\}$ as follows: Write down a solid arrow $i\to \overline{j}$ and a dashed arrow $\begin{tikzcd}i\arrow[dashed]{r} &j\end{tikzcd}$ for each basis vector of $\rad(P_i,P_j)$ in the fixed basis. Let the differential $\partial$ be given by the $k$-dual of the multiplication on $\rad A$. Then, the category $\mathcal{P}^1(A)$ is equivalent to the category of representations over this differential biquiver.
\end{proposition}

In practice, one can thus determine the representation type of an algebra by first computing the differential biquiver as indicated in the foregoing proposition and then apply the reduction algorithm.

\subsection{Examples}
In this section we give several examples for the reduction algorithm applied to bocses arising from strongly quasi-hereditary algebras.

\begin{example}
We start by using the reduction algorithm to construct the well-known Auslander algebra of the path algebra of $kA_3=k(\begin{tikzcd}1\arrow{r}{a}&2\arrow{r}{b}&3\end{tikzcd})$.\\
In the first step, one of the minimal edges $a$ or $b$ can be reduced. We choose $a$ and arrive at the following differential biquiver:
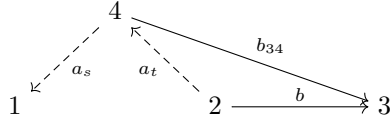
\[\begin{tikzcd}
&4\arrow[dashed]{ld}{a_s}\arrow{rrrd}{b_{34}}\\
1&&2\arrow[dashed]{ul}{a_t}\arrow{rr}{b}&&3
\end{tikzcd}\]
with differential $\partial(b)=-b_{34}a_s$. There are still no superfluous edges, but only one minimal edge $b_{34}$ which we reduce next:
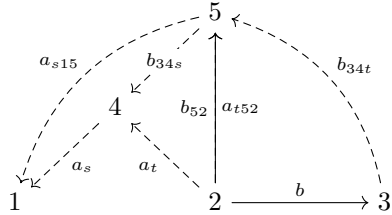
\[
\begin{tikzcd}
&&5\arrow[dashed, bend right]{lldd}[swap]{a_{s15}}\arrow[dashed]{ld}[description]{b_{34s}}\\
&4\arrow[dashed]{ld}{a_s}\\
1&&2\arrow[dashed]{lu}{a_t}\arrow[xshift=-0.5ex]{uu}{b_{52}}\arrow[xshift=0.5ex, dashed]{uu}[swap]{a_{t52}}\arrow{rr}{b}&&3\arrow[dashed, bend right]{lluu}[swap]{b_{34t}}
\end{tikzcd}
\]
with differential $\partial(a_t)=b_{34s}a_{t52}$, $\partial(b_{52})=-a_{t52}+b_{34t}b$, and $\partial(a_{s15})=-a_sb_{34s}$. The next step is to remove the superfluous edge $b_{52}$. This yields the following:
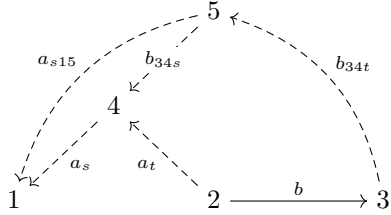
\[
\begin{tikzcd}
&&5\arrow[dashed, bend right]{lldd}[swap]{a_{s15}}\arrow[dashed]{ld}[description]{b_{34s}}\\
&4\arrow[dashed]{ld}{a_s}\\
1&&2\arrow[dashed]{lu}{a_t}\arrow{rr}{b}&&3\arrow[dashed, bend right]{lluu}[swap]{b_{34t}}
\end{tikzcd}
\]
with differential $\partial(a_t)=b_{34s}b_{34t}b$ and $\partial(a_{s15})=-a_sb_{34s}$. As a last step, the minimal edge $b$ is removed:
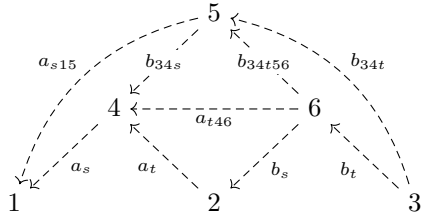
\[
\begin{tikzcd}
&&5\arrow[dashed, bend right]{lldd}[swap]{a_{s15}}\arrow[dashed]{ld}[description]{b_{34s}}\\
&4\arrow[dashed]{ld}{a_s}&&6\arrow[dashed]{dl}{b_s}\arrow[dashed]{ll}{a_{t46}}\arrow[dashed]{ul}[description]{b_{34t56}}\\
1&&2\arrow[dashed]{lu}{a_t}&&3\arrow[dashed, bend right]{lluu}[swap]{b_{34t}}\arrow[dashed]{ul}{b_t}
\end{tikzcd}
\]
This gives a basis of the Jacobson radical of the Auslander algebra where the differential coincides with the dual of the multiplication. It is given by $\partial(b_{34t})=-b_{34t56}b_t$, $\partial(a_{s15})=-a_sb_{34s}$ and $\partial(a_{t46})=-a_tb_s+b_{34s}b_{34t56}$. It is possible to extract the Auslander--Reiten quiver from this data as the irreducible maps are precisely those with vanishing differential. In the example we get the following well-known picture where the dotted arrows represent the Auslander--Reiten translation:
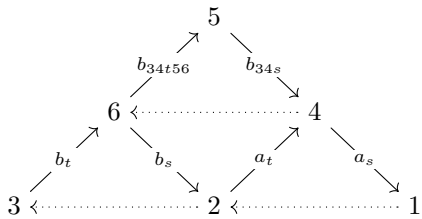
\[
\begin{tikzcd}
&&5\arrow{rd}[description]{b_{34s}}\\
&6\arrow{ru}[description]{b_{34t56}}\arrow{rd}[description]{b_s}&&4\arrow{rd}[description]{a_s}\arrow[dotted]{ll}\\
3\arrow{ru}[description]{b_t}&&2\arrow[dotted]{ll}\arrow{ru}[description]{a_t}&&1\arrow[dotted]{ll}
\end{tikzcd}
\]
\end{example}

\begin{example}\label{sl2example}
Let $\mathfrak{g}=\mathfrak{sl}_2$. The quasi-hereditary algebra $A_\chi$ corresponding to the non-semisimple blocks of BGG category $\mathcal{O}$ was described in Example \ref{examplesl2}. In this example, we use the reduction algorithm to compute the Auslander--Reiten quiver of the category of modules which can be filtered by standard modules. The indecomposable projective modules for $A_\chi$ look as follows:
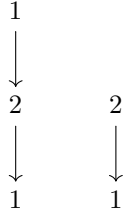
\[\begin{tikzcd}
1\arrow{d}\\
2\arrow{d}&2\arrow{d}\\
1&1
\end{tikzcd}\]
It is easy to see, that the corresponding bocs (constructed in the proof of Theorem  \ref{KoenigKulshammerOvsienko}, see \cite[Appendix A.1]{KoenigKulshammerOvsienko-Quasihereditaryalgebras-2014}) arises from the following differential biquiver:
\[\begin{tikzcd}
1\arrow[yshift=0.5ex]{r}{a}\arrow[yshift=-0.5ex, dashed]{r}[swap]{\varphi}&2
\end{tikzcd}
\]
with zero differential. In this case the corresponding right algebra is basic. We now want to construct the corresponding category of filtered modules using the reduction algorithm.
The only operation we can perform is thus minimal edge reduction at $a$. This yields the following differential biquiver:
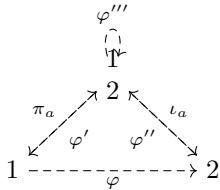
\[\begin{tikzcd}
&{\begin{matrix}1\\2\end{matrix}}\arrow[dashed, yshift=-0.5ex]{rd}[swap]{\varphi''}\arrow[dashed, yshift=0.5ex]{ld}[swap]{\pi_a}\arrow[dashed, loop above]{}{\varphi'''}\\
1\arrow[dashed]{rr}[swap]{\varphi}\arrow[dashed, yshift=-0.5ex]{ru}[swap]{\varphi'}&&2\arrow[dashed, yshift=0.5ex]{lu}[swap]{\iota_a}
\end{tikzcd}\]
with differential given by $\partial(\varphi'')=-\varphi\pi_a$, $\partial(\varphi''')=\iota_a\varphi''-\varphi'\pi_a$, $\partial(\varphi')=\iota_a\varphi$.\\
In this quiver a number $j$ corresponds to a subfactor $\Delta(j)$ in a $\Delta$-filtration of the corresponding module, i.e. $1$ corresponds to $\Delta(1)=L(1)$, $2$ corresponds to $\Delta(2)=P(2)$ and $\begin{matrix}1\\2\end{matrix}$ corresponds to the unique non-split extension of $\Delta(1)$ by $\Delta(2)$. 
Removing the homomorphisms with non-zero differential (i.e. the reducible ones) and rewriting the modules in terms of simple modules we obtain the following Auslander--Reiten quiver of $\mathcal{F}(\Delta)$:
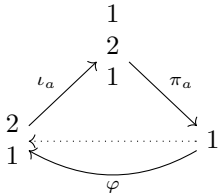
\[
\begin{tikzcd}
&{\begin{matrix}1\\2\\1\end{matrix}}\arrow{rd}{\pi_a}\\
{\begin{matrix}2\\1\end{matrix}}\arrow{ru}{\iota_a}&&{\begin{matrix}1\end{matrix}}\arrow[dotted]{ll}\arrow[bend left]{ll}{\varphi}
\end{tikzcd}
\]
The bend arrow corresponds to an irreducible map which is not part of any Auslander--Reiten sequence. According to results of Krebs \cite{Krebs-AuslanderReitentheory-2015} such morphisms can only occur from injective to projective objects in $\mathcal{F}(\Delta)$.
\end{example}

The above example is a particular case of the following classification of quasi-hereditary algebras with two simple modules, which was first obtained by Mem\-brillo-Her\-n\'an\-dez in \cite{MembrilloHernandez-Quasihereditaryalgebras-1994} (see also \cite{Dubnov-Onderivedcategories-2000, LiuYang-Stratifications-2013} for further study of these algebras from the point of view of derived categories).

\begin{theorem}\label{classification2simples}
Let $A$ be a quasi-hereditary algebra with two simple modules. Then $A$ is Morita equivalent to one of the following path algebras:
\[\begin{tikzcd}
1\arrow[yshift=2.8ex, bend left, color=white]{r}[description]{\color{black}\vdots}\arrow[yshift=4.1ex, bend left]{r}[description]{\alpha_1}\arrow[yshift=0.1ex, bend left]{r}[description]{\alpha_{s}}&2\arrow[yshift=-1.2ex, bend left, color=white]{l}[description]{\color{black}\vdots}\arrow[yshift=-4.3ex, bend left]{l}[description]{\beta_1}\arrow[yshift=-0.1ex, bend left]{l}[description]{\beta_t}
\end{tikzcd}\]
with relations $\alpha_i\beta_j$ for each $1\leq i\leq s$ and $1\leq j\leq t$. In each of the cases, the subalgebra formed by just taking the $\alpha_i$ is an exact Borel subalgebra.\\
The category of filtered modules is of finite representation-type if and only if $s\leq 1$. It is tame if and only if $s=2$.
\end{theorem}

\begin{proof}
Since (up to Morita equivalence) each quasi-hereditary algebra is the right algebra of a bocs, it suffices to classify those. It is clear that a directed algebra on two simples can only be hereditary. Thus, the bocs comes from a differential biquiver. For degree reasons, the only possible non-vanishing differential must lead to a non-regular arrow. Cancelling those, using Proposition \ref{regularisation}, we can assume that the differential vanishes.\\
Thus, the quasi-hereditary algebras on two simples are classified by two numbers, the number $s$ of arrows from $1$ to $2$ in the quiver of $B$ and the number $t$ of generators for $\overline{W}$ as a $B$-$B$-bimodule in a minimal generating set. Using the differential biquiver description, it is easy to see that the right algebra of the corresponding bocs will have dimension $2+s+t+st$. Since $s$ and $t$ are determined by the homomorphism spaces between the corresponding standard modules as $s=\dim \Ext^1(\Delta(1),\Delta(2))$ and $t=\dim \Hom(\Delta(2), \Delta(1))$, these algebras cannot be Morita equivalent.\\
It is straightforward to check that the corresponding right algebra is basic in each of the cases. Since the differential vanishes, two modules are isomorphic as $B$-modules if and only if they are isomorphic as $\mathfrak{B}$-modules. This readily yields the result on the representation type.
\end{proof}

\begin{example}
The philosophy could be that the closer $R$ is to be basic, the better the induction functor is behaved. The following example, constructed in joint work with Agnieszka Bodzenta \cite{BodzentaKulshammer-Stronglyquasihereditary-2016}, shows that this is not the case. It gives an example of a basic quasi-hereditary algebra $R$ with representation-finite filtered category but tame exact Borel subalgebra $B$. After constructing it, we learned that this algebra was already used as a counterexample in a different context by Mazorchuk, see \cite{Mazorchuk-Koszulduality-2010}.\\
The algebra is given by the following quiver:
\[
\begin{tikzcd}
3\arrow[yshift=0.5ex]{r}{\alpha}&1\arrow[yshift=-0.5ex]{l}{\beta}\arrow[yshift=0.5ex]{r}{\gamma}&2\arrow[yshift=-0.5ex]{l}{\delta}
\end{tikzcd}
\]
with relations $\beta\alpha$, $\gamma\delta$, and $\beta\delta\gamma\alpha$.
Its projective modules are given as follows:
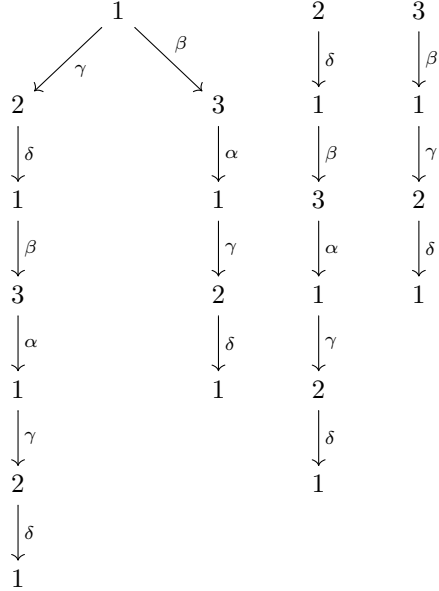
\[\begin{tikzcd}
&1\arrow{ld}{\gamma}\arrow{rd}{\beta}&&2\arrow{d}{\delta}&3\arrow{d}{\beta}\\
2\arrow{d}{\delta}&&3\arrow{d}{\alpha}&1\arrow{d}{\beta}&1\arrow{d}{\gamma}\\
1\arrow{d}{\beta}&&1\arrow{d}{\gamma}&3\arrow{d}{\alpha}&2\arrow{d}{\delta}\\
3\arrow{d}{\alpha}&&2\arrow{d}{\delta}&1\arrow{d}{\gamma}&1\\
1\arrow{d}{\gamma}&&1&2\arrow{d}{\delta}\\
2\arrow{d}{\delta}&&&1\\
1
\end{tikzcd}\]
The minimal projective resolutions of standard modules are hence of the following form:
\[\begin{tikzcd}
0\arrow{r}&P_2\oplus P_3\arrow{r}&P_1\arrow{r}&\Delta(1)\arrow{r}&0\\
0\arrow{r}&P_3\arrow{r}&P_2\arrow{r}&\Delta(2)\arrow{r}&0\\
&0\arrow{r}&P_3\arrow{r}&\Delta(3)\arrow{r}&0
\end{tikzcd}\]
The $\Ext^*$-algebra can be represented by the following differential biquiver (without relations):
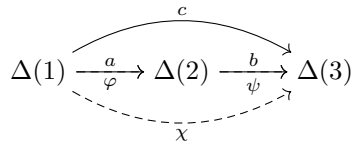
\[\begin{tikzcd}
\Delta(1)\arrow[yshift=0.5ex]{r}{a}\arrow[yshift=-0.5ex, dashed]{r}[swap]{\varphi}\arrow[yshift=0.5ex, bend left]{rr}{c}\arrow[yshift=-0.5ex, bend right, dashed]{rr}[swap]{\chi}&\Delta(2)\arrow[yshift=0.5ex]{r}{b}\arrow[yshift=-0.5ex, dashed]{r}[swap]{\psi}&\Delta(3)
\end{tikzcd}\]
and differential $\partial(\chi)=\psi\varphi$ and $\partial(c)=b\varphi$. The exact Borel subalgebra is given by a tame hereditary algebra $k\tilde{A}_2$, where $\tilde{A}_2$ is the Euclidean quiver of type $A$ with $3$ vertices. Computing $\End_{\mathfrak{B}}(B)$, one readily sees that its dimension is the same as the dimension of $A$. In particular it is basic.\\
The following diagram shows that all the modules in the $1$-parameter family of three-dimensional modules for $k\tilde{A}_2$ become isomorphic in the category of modules over the bocs:
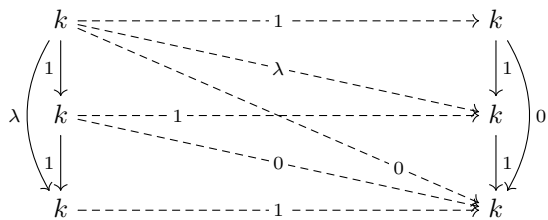
\[
\begin{tikzcd}[column sep=35ex]
k\arrow[bend right]{dd}[swap]{\lambda}\arrow{d}[swap]{1}\arrow[dashed]{r}[description]{1}\arrow[dashed]{rd}[description]{\lambda}\arrow[dashed]{rdd}[pos=0.8, description]{0}&k\arrow[bend left]{dd}{0}\arrow{d}{1}\\
k\arrow{d}[swap]{1}\arrow[dashed]{r}[pos=0.25, description]{1}\arrow[dashed]{dr}[description]{0}&k\arrow{d}{1}\\
k\arrow[dashed]{r}[description]{1}&k
\end{tikzcd}
\]
The following table lists the reductions applied to arrive at a bocs with no solid arrows as well as the number of vertices and arrows in each step to give the reader an impression of the growth rate:
\begin{center}
\begin{tabular}{|l|c|c|}
\hline
Step &number of vertices &number of arrows\\
\hline
start & $3$ & $6$\\
minimal edge reduction at $a$ & $4$ &$14$\\  
minimal edge reduction at $b_{34}$ & $5$ & $35$\\
minimal edge reduction at $b$ & $6$ & $55$\\
regularisation at $b_{52}$ & $6$ & $53$\\
regulariation at $b_{52,56}$ &$6$ &$51$\\
minimal edge reduction at $c$ & $7$ & $75$\\
regularisation at $c_{61}$ & $7$ &$73$\\
regularisation at $c_{51}$ & $7$ & $71$\\
regularisation at $c_{61,67}$ &$7$&$69$\\
regularisation at $c_{51,57}$ &$7$&$67$\\
minimal edge reduction at $c_{34}$ &$8$&$107$\\
regularisation at $c_{34,74}$ &$8$&$105$\\
regularisation at $c_{34,64}$ &$8$&$103$\\
regularisation at $c_{34,53}$ &$8$&$101$\\
regularisation at $c_{34,74,78}$ & $8$ & $99$\\
regularisation at $c_{34,64,68}$ & $8$ &$97$\\
regularisation at $c_{34,54,58}$ & $8$ & $95$\\
minimal edge reduction at $c_{34,35}$ & $9$ & $156$\\
regularisation at $c_{34,35,85}$ & $9$ & $154$\\
regularisation at $c_{34,35,75}$ & $9$ & $152$\\
regularisation at $c_{34,35,65}$ & $9$ & $150$\\
regularisation at $c_{34,55,95}$ & $9$ & $148$\\
regularisation at $c_{34,55}$ & $9$ & $146$\\
regularisation at $c_{34,35,85,89}$ & $9$ & $144$\\
regularisation at $c_{34,35,75,79}$ & $9$ & $142$\\
regularisation at $c_{34,35,65,69}$ & $9$ & $140$\\
regularisation at $c_{34,55,99}$ & $9$ & $138$\\
regularisation at $c_{34,55,59}$ & $9$ & $136$\\
\hline
\end{tabular}
\end{center}
Removing the reducible maps yields the following Auslander--Reiten quiver of $\mathcal{F}(\Delta)$. As in Example \ref{sl2example} the numbers $1,2,3$ correspond to the modules $\Delta(1)$, $\Delta(2)$, $\Delta(3)$, respectively. Stacking numbers on top of each other corresponds to an extension of the upper module by the lower module.
\[
\begin{tikzcd}[ampersand replacement=\&]
3\arrow{rd}[description]{c_{34,35,t}}\&\& {\begin{smallmatrix}1\\2\\3\end{smallmatrix}}\arrow{rd}[description]{c_{34,s}}\\
\&{\begin{smallmatrix}&1\\2&&3\\3\end{smallmatrix}}\arrow{ru}[description]{c_{34,35,s}}\arrow{rd}[description]{c_{34,t,89}}\&\&{\begin{smallmatrix}1\\2\end{smallmatrix}}\arrow{rd}[description]{a_s}\arrow[bend right=70]{lllu}[description]{b_{34,s}}\\
{\begin{smallmatrix}2\\3\end{smallmatrix}}\arrow{ru}[description]{b_{34,t,56,96}}\arrow{rd}[description]{b_s}\&\&{\begin{smallmatrix}&1\\2&&3\end{smallmatrix}}\arrow{ru}[description]{c_{34,s}}\arrow{rd}[description]{c_{t,78}}\&\&1\arrow{rd}[description]{\varphi}\\
\&2\arrow{ru}[description]{a_{t,82}}\&\&{\begin{smallmatrix}1\\3\end{smallmatrix}}\arrow{ru}[description]{c_s}\arrow{rd}[description]{b_{t,67}}\&\&2\\
\&\&\&\&{\begin{smallmatrix}2\\3\end{smallmatrix}}\arrow{ru}[description]{b_s}
\end{tikzcd}
\]
The two modules labelled $\begin{smallmatrix}2\\3\end{smallmatrix}$ and the two modules labelled $2$ have to be identified, respectively.
\end{example}

\section{Bocses of Schur algebras}

In this section we will compute the bocses of the representation-finite and the tame Schur algebras and prove the result that the tame Schur algebras have a category of modules filtered by Weyl modules of finite representation type. This is joint work with Ulrich Thiel \cite{KulshammerThiel-Boxes-2016}.

\subsection{The bocses of the representation-finite Schur algebras}

In this subsection, the bocses associated to the representation-finite Schur algebras are described. Using the construction explained in \cite{KoenigKulshammerOvsienko-Quasihereditaryalgebras-2014} (which we referred to in Theorem \ref{KoenigKulshammerOvsienko}) in each case we obtain a bocs whose right algebra is Morita equivalent to the Schur algebra we started with and which has an exact Borel subalgebra.
This example shows that the dimension of the exact Borel subalgebra and of the right algebra can differ a lot from the dimension of the corresponding basic quasi-hereditary algebra.\\
A presentation of the corresponding basic algebras by quiver and relations was given in Theorem  \ref{representationtypeSchur}. The indecomposable projective modules for the algebra $(\mathcal{A}_n)$ are given as follows where $1<i<n$:
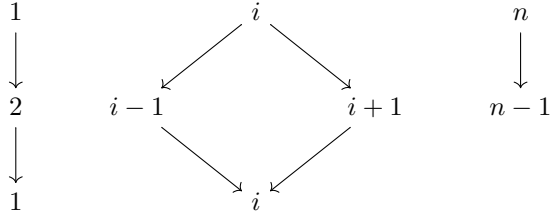
\[
\begin{tikzcd}
1\arrow{d}&&i\arrow{ld}\arrow{rd}&&n\arrow{d}\\
2\arrow{d}&i-1\arrow{rd}&&i+1\arrow{ld}&n-1\\
1&&i
\end{tikzcd}
\]
The dimension of the algebra is then given by $4(n-1)+1$. The projective resolutions of all standard modules are "contained" in the following projective resolution of $\Delta(1)$, i.e. all standard modules are syzygies of $\Delta(1)$.
\[
\begin{tikzcd}
0\arrow{r}&P(n)\arrow{r}&\dots\arrow{r}&P(2)\arrow{r}&P(1)\arrow{r}&\Delta(1)\arrow{r}&0
\end{tikzcd}
\]
The $A_\infty$-structure on the Ext-algebra of the standard modules has been computed by Klamt and Stroppel in \cite{KlamtStroppel-OntheExtalgebras-2012}. Translating to the language of bocses one arrives at the following:
\[
\begin{tikzcd}
1\arrow[yshift=0.5ex]{r}{a_1}\arrow[yshift=0.5ex, bend left]{rr}{b_1}\arrow[yshift=-0.5ex, dashed]{r}[swap]{\varphi_1}&2\arrow[yshift=0.5ex, bend left]{rr}{b_2}\arrow[yshift=0.5ex]{r}{a_2}\arrow[yshift=-0.5ex, dashed]{r}[swap]{\varphi_2}&3\arrow[yshift=0.5ex]{r}{a_3}\arrow[yshift=-0.5ex, dashed]{r}[swap]{\varphi_3}&\dots\arrow[yshift=0.5ex]{r}{a_{n-1}}\arrow[yshift=-0.5ex]{r}[swap]{\varphi_{n-1}}&n
\end{tikzcd}
\]
with relations $a_{i+1}a_i$ and $a_{i+2}b_i+b_{i+1}a_i$ and non-zero differential $\partial(b_i)=\varphi_{i+1}a_i-a_{i+1}\varphi$. Let $B$ be the path algebra of the solid part with the relations $a_{i+1}a_i$ and $a_{i+2}b_i+b_{i+1}a_i$. It is easy to see that the projective $B$-module to the vertex $i$ has dimension $n-i+1$. So in total, $B$ has dimension $\sum_{j=1}^n j=\dfrac{n(n+1)}{2}$. In particular, it is not bound in terms of the dimension of $A$.\\
For computing the dimension of the right algebra $R$ of the associated bocs $(B,W)$, one can use the fact that $A$ has a duality. By \cite[Corollary 4.2]{Koenig-ExactBorelsubalgebras2-1995} the dimension of $R$ is in this case given by 
\[\sum_{i=1}^n(\dim P_B(i))^2=\sum_{j=1}^n j^2=\dfrac{n(n+1)(2n+1)}{6}.\] 

\subsection{The bocses of the tame Schur algebras}

In this subsection, the bocses of the blocks of the tame Schur algebras $(\mathcal{D}_3)$, $(\mathcal{D}_4)$, $(\mathcal{R}_4)$ and $(\mathcal{H}_4)$ defined in Theorem \ref{representationtypeSchur} are computed. We start with $(\mathcal{D}_3)$. The indecomposable projective modules for the algebra $(\mathcal{D}_3)$ are given as follows:

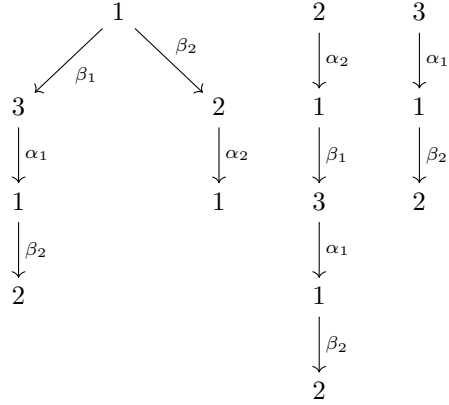
\[
\begin{tikzcd}
&1\arrow{ld}{\beta_1}\arrow{rd}{\beta_2}&&2\arrow{d}{\alpha_2}&3\arrow{d}{\alpha_1}\\
3\arrow{d}{\alpha_1}&&2\arrow{d}{\alpha_2}&1\arrow{d}{\beta_1}&1\arrow{d}{\beta_2}\\
1\arrow{d}{\beta_2}&&1&3\arrow{d}{\alpha_1}&2\\
2&&&1\arrow{d}{\beta_2}\\
&&&2\\
\end{tikzcd}
\]
The projective resolutions of standard modules are of the following form:
\[
\begin{tikzcd}
0\arrow{r} &P_3\arrow{r} &P_2\oplus P_3\arrow{r} &P_1\arrow{r} &\Delta(1)\arrow{r}&0\\
&0\arrow{r}&P_3\arrow{r}&P_2\arrow{r}&\Delta(2)\arrow{r}&0\\
&&0\arrow{r}&P_3\arrow{r}&\Delta(3)\arrow{r}&0
\end{tikzcd}
\]
It is clear that there are no higher multiplications in the $A_\infty$-category $\Ext^*(\Delta,\Delta)$ since there is no path of length $\geq 3$ in the $\Ext^*$-quiver. We obtain the corresponding bocs easily as the following differential biquiver
\[
\begin{tikzcd}
\Delta(1)\arrow[yshift=0.5ex]{r}{a}\arrow[yshift=-0.5ex, dashed]{r}[swap]{\varphi}\arrow[yshift=0.5ex, bend left]{rr}{c} &\Delta(2)\arrow[yshift=0.5ex]{r}{b}\arrow[yshift=-0.5ex, dashed]{r}[swap]{\psi}&\Delta(3)
\end{tikzcd}
\]
with relation $ba$ and $\partial(c)=-b\varphi$.

For the algebra $(\mathcal{R}_4)$, the indecomposable projective modules are given as follows:

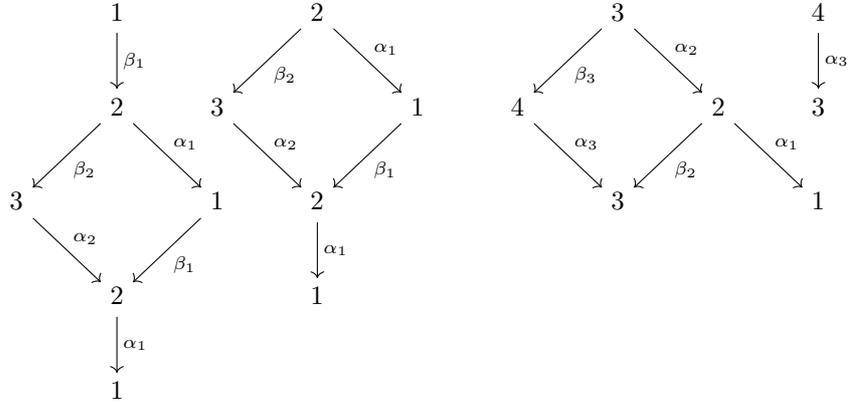
\[
\begin{tikzcd}
&1\arrow{d}{\beta_1}&&2\arrow{ld}{\beta_2}\arrow{rd}{\alpha_1}&&&3\arrow{ld}{\beta_3}\arrow{rd}{\alpha_2}&&4\arrow{d}{\alpha_3}\\
&2\arrow{ld}{\beta_2}\arrow{rd}{\alpha_1}&3\arrow{rd}{\alpha_2}&&1\arrow{ld}{\beta_1}&4\arrow{rd}{\alpha_3}&&2\arrow{ld}{\beta_2}\arrow{rd}{\alpha_1}&3\\
3\arrow{rd}{\alpha_2}&&1\arrow{ld}{\beta_1}&2\arrow{d}{\alpha_1}&&&3&&1\\
&2\arrow{d}{\alpha_1}&&1\\
&1
\end{tikzcd}
\]
Hence, the projective resolutions of its standard modules are given as follows:
\[
\begin{tikzcd}
&0\arrow{r}&P_2\arrow{r}&P_1\arrow{r}&\Delta(1)\arrow{r}&0\\
0\arrow{r}&P_4\arrow{r}&P_3\arrow{r}&P_2\arrow{r}&\Delta(1)\arrow{r}&0\\
&0\arrow{r}&P_4\arrow{r}&P_3\arrow{r}&\Delta(3)\arrow{r}&0\\
&&0\arrow{r}&P_4\arrow{r}&\Delta(4)\arrow{r}&0
\end{tikzcd}
\]
Here, there are paths of length $3$ in the $\Ext^*$-quiver, but $\Ext^*_A(\Delta(1),\Delta(4))=0$, i.e. there is no possible non-zero result of such a multiplication. Thus, the $\Ext$-algebra of the standard modules is formal again, and the corresponding bocs is given by the following biquiver 
\[
\begin{tikzcd}
\Delta(1)\arrow[yshift=0.5ex]{r}{a}\arrow[yshift=-0.5ex, dashed]{r}[swap]{\varphi}\arrow[yshift=0.5ex, bend left]{rr}{c}\arrow[yshift=-0.5ex, dashed, bend right]{rr}[swap]{\chi} &\Delta(2)\arrow[yshift=0.5ex, bend left]{rr}{e}\arrow[yshift=0.5ex]{r}{b}\arrow[yshift=-0.5ex, dashed]{r}[swap]{\psi}&\Delta(3)\arrow[yshift=0.5ex]{r}{d}\arrow[yshift=-0.5ex, dashed]{r}[swap]{\rho}&\Delta(4)
\end{tikzcd}
\]
with relation $db$ and $\partial(\chi)=\psi\varphi$, $\partial(c)=\psi a-b\varphi$, and $\partial(e)=d\psi-\rho b$.

For the algebra $(\mathcal{H}_4)$, the indecomposable projective modules cannot be that nicely pictured as in the previous cases. The following is an attempt to do so anyway. For the precise meaning, the reader should consult the relations:
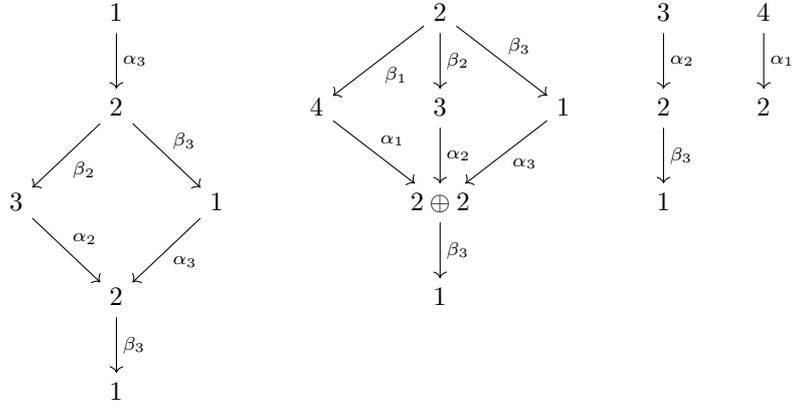
\[
\begin{tikzcd}
&1\arrow{d}{\alpha_3}&&&2\arrow{ld}{\beta_1}\arrow{d}{\beta_2}\arrow{rd}{\beta_3}&&3\arrow{d}{\alpha_2}&4\arrow{d}{\alpha_1}\\
&2\arrow{ld}{\beta_2}\arrow{rd}{\beta_3}&&4\arrow{rd}{\alpha_1}&3\arrow{d}{\alpha_2}&1\arrow{ld}{\alpha_3}&2\arrow{d}{\beta_3}&2\\
3\arrow{rd}{\alpha_2}&&1\arrow{ld}{\alpha_3}&&2\oplus 2\arrow{d}{\beta_3}&&1\\
&2\arrow{d}{\beta_3}&&&1\\
&1
\end{tikzcd}
\]
The projective resolutions of the standard modules are given as follows:
\[
\begin{tikzcd}
0\arrow{r}&P_4\arrow{r}&P_2\arrow{r}&P_1\arrow{r}&\Delta(1)\arrow{r}&0\\
&0\arrow{r}&P_3\oplus P_4\arrow{r}&P_2\arrow{r}&\Delta(2)\arrow{r}&0\\
&&0\arrow{r}&P_3\arrow{r}&\Delta(3)\arrow{r}&0\\
&&0\arrow{r}&P_4\arrow{r}&\Delta(4)\arrow{r}&0
\end{tikzcd}
\]
Computing the $\Ext$-algebra of the standard modules, we again notice that there is no path of length greater or equal to three. Thus, this $\Ext$-algebra is formal. The differential biquiver is as follows:
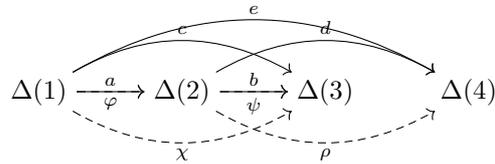
\[
\begin{tikzcd}
\Delta(1)\arrow[yshift=0.5ex]{r}{a}\arrow[yshift=0.5ex, bend left]{rr}{c}\arrow[yshift=1.5ex, bend left]{rrr}{e}\arrow[dashed, yshift=-0.5ex]{r}[swap]{\varphi}\arrow[yshift=-0.5ex, bend right, dashed]{rr}[swap]{\chi}&\Delta(2)\arrow[yshift=0.5ex]{r}{b}\arrow[yshift=-0.5ex, dashed]{r}[swap]{\psi}\arrow[yshift=0.5ex, bend left]{rr}{d}\arrow[yshift=-0.5ex, bend right, dashed]{rr}[swap]{\rho}&\Delta(3)&\Delta(4)
\end{tikzcd}
\]
with relation $da$ and differential $\partial(\chi)=\psi\varphi$, $\partial(c)=\psi a-b\varphi, \partial(e)=\rho a-d\varphi$.

For the algebra $(\mathcal{D}_4)$, the indecomposable projective modules are given as follows:
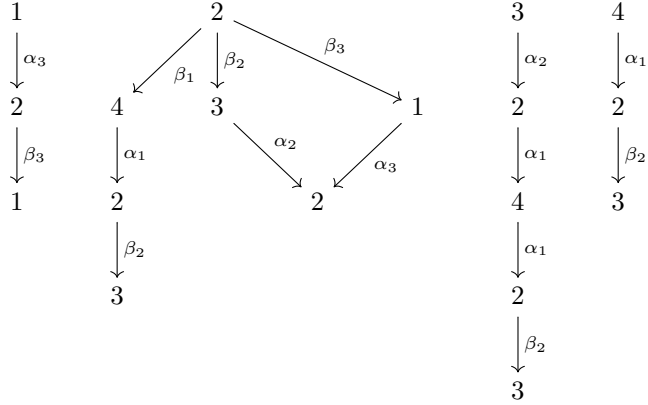
\[
\begin{tikzcd}
1\arrow{d}{\alpha_3}&&2\arrow{ld}{\beta_1}\arrow{d}{\beta_2}\arrow{rrd}{\beta_3}&&&3\arrow{d}{\alpha_2}&4\arrow{d}{\alpha_1}\\
2\arrow{d}{\beta_3}&4\arrow{d}{\alpha_1}&3\arrow{rd}{\alpha_2}&&1\arrow{ld}{\alpha_3}&2\arrow{d}{\alpha_1}&2\arrow{d}{\beta_2}\\
1&2\arrow{d}{\beta_2}&&2&&4\arrow{d}{\alpha_1}&3\\
&3&&&&2\arrow{d}{\beta_2}\\
&&&&&3
\end{tikzcd}
\]
The projective resolutions of its standard modules take the following form:
\[
\begin{tikzcd}
0\arrow{r}&P_4\arrow{r}&P_3\oplus P_4\arrow{r} &P_2\arrow{r}&P_1\arrow{r}&\Delta(1)\arrow{r}&0\\
&0\arrow{r}&P_4\arrow{r}&P_3\oplus P_4\arrow{r}&P_2\arrow{r}&\Delta(2)\arrow{r}&0\\
&&0\arrow{r}&P_4\arrow{r}&P_3\arrow{r}&\Delta(3)\arrow{r}&0\\
&&&0\arrow{r}&P_4\arrow{r}&\Delta(4)\arrow{r}&0
\end{tikzcd}
\]
Computing the $\Ext^*$-algebra, we arrive at the following biquiver:
\[
\begin{tikzcd}
\Delta(1)\arrow[yshift=0.5ex]{r}{a}\arrow[yshift=0.5ex, bend left]{rr}{c}\arrow[yshift=-0.5ex, dashed]{r}[swap]{\varphi}&\Delta(2)\arrow[yshift=0.5ex]{r}{b}\arrow[dashed, yshift=-0.5ex]{r}[swap]{\psi}\arrow[yshift=0.5ex, bend left]{rr}{e}&\Delta(3)\arrow[yshift=0.5ex]{r}{d}\arrow[yshift=-0.5ex, dashed]{r}[swap]{\chi}&\Delta(4)
\end{tikzcd}
\]
We claim the following relations and differential on it: $db, ba, ea+dc$ and $\partial(c)=\psi a-b\varphi, \partial(e)=d\psi$.\\
The only possible non-zero higher multiplication is $m_3(d,b,a)$ since $m_3$ is of degree $-1$ and has to have as input a path of length $3$, and $\Ext^*_A(\Delta(1),\Delta(4))$ is concentrated in degree $2$. We use Merkulov's\label{exampleMerkulov} construction (see Remark \ref{Merkulov}), to this that also this higher multiplication vanishes. Writing down $d$, $b$ and $a$ in one diagram, one sees that $\lambda_2$ of every two of them lands in $\mathcal{H}$. Hence, $G$ vanishes on them. Thus $\lambda_3(d,b,a)=(\lambda_2(G\lambda_2\otimes 1)-\lambda_2(1\otimes G\lambda_2))(d,b,a)=0$.
\[
\begin{tikzcd}
0\arrow{r}&P_4\arrow{d}{\id}\arrow{r}&P_3\oplus P_4\arrow{d}{\id}\arrow{r}&P_2\arrow{d}{\id}\arrow{r}&P_1\arrow{r}&0\\
0\arrow{r}&P_4\arrow{d}{\id}\arrow{r}&P_3\oplus P_4\arrow{d}{(\id,0)}\arrow{r}&P_2\arrow{d}{0}\arrow{r}&0\\
0\arrow{r}&P_4\arrow{d}{\id}\arrow{r}&P_3\arrow{d}{0}\arrow{r}&0\\
0\arrow{r}&P_4\arrow{r}&0
\end{tikzcd}
\]

\subsection{Filtered representation type of tame Schur algebras}

In this subsection we will prove the second main theorem of this paper:

\begin{theorem}
The subcategory $\mathcal{F}(\Delta)$ for the quasi-hereditary algebras $(\mathcal{D}_3)$, $(\mathcal{D}_4)$, $(\mathcal{H}_4)$ and $(\mathcal{R}_4)$ given in Theorem \ref{representationtypeSchur} is of finite representation type, independent of the characteristic. In particular, $\mathcal{F}(\Delta)$ is of finite representation type for the tame Schur algebras.
\end{theorem}

\begin{proof}
Let $A$ be one of the algebras mentioned in the theorem. In the previous section for each of these a bocs $\mathfrak{B}=(B,W)$ with $\modu \mathfrak{B}\cong \mathcal{F}(\Delta)$ was constructed. For the algebras $(\mathcal{D}_3)$, respectively $(\mathcal{D}_4)$, the algebra $B$ is the following: $\begin{tikzcd}1\arrow{r}{a}\arrow[yshift=0.5ex, bend left]{rr}{c}&2\arrow{r}{b}&3\end{tikzcd}$ with relation $ba$, respectively $\begin{tikzcd}1\arrow{r}{a}\arrow[yshift=0.5ex, bend right]{rr}[swap]{c}&2\arrow[yshift=0.5ex,bend left]{rr}{e}\arrow{r}{b}&3\arrow{r}{d}&4\end{tikzcd}$ with relations $db$, $ba$, and $ea+dc$. In both of these cases, this is a representation-finite special biserial algebra (see e.g. \cite{WaldWaschbusch-Tamebiserial-1985}). By Corollary \ref{exactBorelfinitetype}, this implies that also $\mathcal{F}(\Delta)$ is of finite representation type.\\
For the algebras $(\mathcal{R}_4)$ and $(\mathcal{H}_4)$ this approach is not possible since the corresponding exact Borel subalgebras are not of finite representation type. To prove the claim, the reduction algorithm, in the version of Bautista and Salmer\'on, see Proposition \ref{BautistaSalmeronreduction}, is applied:\\
For $(\mathcal{H}_4)$ the first step is to do a minimal edge reduction at $a$. Then $F(da)=(d,d_{4,5})\begin{pmatrix}0&0\\1&0\end{pmatrix}=(d_{4,5},0)$. Hence the solid edge $d_{4,5}$ can be removed as a next step. The resulting bocs has no relations anymore, so the standard reduction algorithm for differential biquivers can be applied. The following table only lists the minimal edge reductions leaving out the regularisations which are done inbetween:
\begin{center}
\begin{tabular}{|l|c|c|}
\hline
step & number of vertices & number of arrows\\
\hline
start & $4$ & $9$\\
minimal edge reduction at $a$ & $5$ & $20$\\
removing $d_{4,5}$ & $5$ & $19$\\
minimal edge reduction at $b_{3,5}$ &$6$ &$42$\\
minimal edge reduction at $b$ & $7$ &$64$\\
minimal edge reduction at $c$ &$8$ &$85$\\
minimal edge reduction at $d$ &$9$ &$83$\\
minimal edge reduction at $e$ &$10$ &$105$\\
minimal edge reduction at $d_{4,7}$ &$11$ &$139$\\
minimal edge reduction at $e_{4,8}$ &$12$ &$180$\\
minimal edge reduction at $d_{47,107}$ &$13$ &$228$\\
regularisations &$13$ &$194$\\
\hline
\end{tabular}
\end{center}
The Auslander--Reiten quiver in this case has $13$ vertices and $20$ edges.\\
Similarly, in the case of $(\mathcal{R}_4)$, one first reduces $b$ and cancels $d_{4,5}$ to arrive at a differential biquiver without relations to which the standard reduction algorithm for differential biquivers can be applied. The corresponding table looks as follows:
\begin{center}
\begin{tabular}{|l|c|c|}
\hline
step & number of vertices & number of arrows\\
\hline
start &$4$ &$9$\\
minimal edge reduction at $b$ & $5$ & $20$\\
removing $d_{4,5}$ & $5$ & $19$\\
minimal edge reduction at $a_{51}$ &$6$ & $42$\\
minimal edge reduction at $a$ &$7$ & $63$\\
minimal edge reduction at $c$ &$8$ & $85$\\
minimal edge reduction at $d_{4,8}$ &$9$ &$92$\\
minimal edge reduction at $e_{4,7}$ &$10$ &$125$\\
minimal edge reduction at $d$ &$11$ &$142$\\
minimal edge reduction at $e$ &$12$ &$170$\\
minimal edge reduction at $d_{10,3}$ &$13$ &$222$\\
regularisations & $13$ &$194$\\
\hline
\end{tabular}
\end{center}
The resulting Auslander--Reiten quiver also has $13$ vertices and $20$ edges.
\end{proof}

\begin{remark}
\begin{enumerate}[(i)]
\item The algebra $(\mathcal{D}_3)$ is itself special-biserial. It is therefore possible to explicitly write down all the indecomposable modules, and see which ones can be filtered by standard modules. Also for the algebra $(\mathcal{D}_4)$ there is a method known to construct all the indecomposable modules up to isomorphism. This method is due to Br\"ustle and based on methods related to bocses. He calls this class of algebras KIT-algebras, see \cite{Bruestle-Kitalgebras-2001}.
\item That the category of modules filtered by standard modules for $(\mathcal{D}_3)$ and $(\mathcal{D}_4)$ is of finite representation type also follows from \cite{ErdmanndelaPenaSaenz-RelativeAuslanderReiten-2002} where this is proved for each of the algebras in the more general $(\mathcal{D}_n)$-series. In particular, for $n\geq 5$ these algebras give examples of wild blocks of Schur algebras such that the category of modules  filtered by standard modules is of finite representation type.
\item For the algebras $(\mathcal{R}_4)$ and $(\mathcal{H}_4)$ one can also use the reduction algorithm of differential biquivers on the differential biquivers computed in the previous subsection, omitting the  relations. For $(\mathcal{R}_4)$ this results in a bocs with $15$ vertices and $276$ dashed edges (the Auslander--Reiten quiver having 15 vertices and 23 edges). For $(\mathcal{H}_4)$ it results in a bocs with $17$ vertices and $452$ edges (the Auslander--Reiten quiver having $17$ vertices and $27$ edges). Since the categories of filtered modules for $(\mathcal{R}_4)$ and $(\mathcal{H}_4)$ are subcategories of the categories of representations of these differential biquivers, it follows that they are representation-finite with less than $15$, respectively $17$, isomorphism classes of indecomposable modules.
\end{enumerate}
\end{remark}

\end{document}